\documentclass[11pt]{amsart}
\usepackage{times}
\usepackage{amsfonts,amssymb,amsmath,amsgen,amsthm,faktor}
\usepackage{hyperref}
\usepackage{color}
\usepackage[utf8]{inputenc}
\usepackage[american]{babel}
\usepackage{cancel}
\usepackage{dsfont}
\usepackage[a4paper]{geometry}

\theoremstyle{plain}
\newtheorem{theorem}{Theorem}[section]

\newtheorem{assumption}[theorem]{Assumption}
\newtheorem{lemma}[theorem]{Lemma}
\newtheorem{corollary}[theorem]{Corollary}
\newtheorem{proposition}[theorem]{Proposition}

\theoremstyle{remark}
\newtheorem{remark}[theorem]{Remark}

\def\R{{\mathbb R}}
\def\N{{\mathbb N}}
\def\Z{{\mathbb Z}}
\def\T{{\mathbb T}}

\def\({\left(}
\def\){\right)}
\def\<{\left\langle}
\def\>{\right\rangle}

\def\Tend#1#2{\mathop{\longrightarrow}\limits_{#1\rightarrow#2}}

\def\eps{\varepsilon}

\def\op{{\rm op}}

\DeclareMathOperator{\loc}{loc}

\DeclareMathOperator{\supp}{supp}

\DeclareMathOperator{\Tr}{Tr}

\newcommand{\To}{\longrightarrow}

\numberwithin{equation}{section}

\begin{document}

\title[]{Wigner measures and effective mass theorems}
\author[V. Chabu]{Victor Chabu}
\thanks{V. Chabu was supported by the grant 2017/13865-0, São Paulo Research Foundation (FAPESP). F. Macià has been supported by grants StG-2777778 (U.E.) and MTM2013-41780-P, MTM2017-85934-C3-3-P, TRA2013-41096-P (MINECO, Spain).}
\address{Universidade de São Paulo, IF-USP, DFMA, CP 66.318 05314-970, São Paulo, SP, Brazil}
\email{vbchabu@if.usp.br}
\author[C. Fermanian]{Clotilde~Fermanian-Kammerer}
\address{LAMA, UMR CNRS 8050,
Universit\'e Paris EST\\
61, avenue du G\'en\'eral de Gaulle\\
94010 Cr\'eteil Cedex\\ France}
\email{Clotilde.Fermanian@u-pec.fr}
\author[F. Maci\`{a}]{Fabricio Maci\`{a}}
\address{Universidad Polit\'{e}cnica de Madrid. ETSI Navales. Avda. de la Memoria 4, 28040 Madrid, Spain}
\email{fabricio.macia@upm.es}

\begin{abstract}
We study a Schr\"odinger equation which describes the dynamics of an electron in a crystal in the presence of impurities. We consider the regime of small wave-lengths comparable to the characteristic scale of  the crystal. 
It is well-known that under suitable assumptions on the initial data and for highly oscillating potentials, the wave function can be approximated  by the solution of a simpler equation, the effective mass equation. 
Using Floquet-Bloch decomposition, as it is classical in this subject, we establish effective mass equations in a rather general setting. In particular,  Bloch bands  are allowed to have  degenerate critical points, as  may occur in dimension strictly larger than one.
Our analysis leads to a new type of effective mass equations which are operator-valued  and of Heisenberg form and relies on Wigner measure theory and, more precisely, to its applications to the analysis of dispersion effects. \\

\end{abstract}

\maketitle

\section{Introduction} 

\subsection{The dynamics of an electron in a crystal and the effective mass equation}
The dynamics of an electron in a crystal in the presence of impurities is described by a wave function $\Psi(t,x)$ that solves the Schrödinger equation:
\begin{equation}\label{eq:schrodbrut}
\left\{
\begin{array}{l}
i\hbar\partial_t \Psi(t,x) +\dfrac{\hbar^2}{2m} \Delta_x \Psi(t,x) - Q_{\rm per}\left(x\right)\Psi(t,x) - Q_{\rm ext}(t,x)\Psi(t,x) =0,\vspace{0.2cm}\\
\Psi |_{t=0}=\Psi_0.
\end{array}\right.
\end{equation}
The potential $Q_{\rm per}$ is periodic with respect to some lattice in $\R^d$ and describes the interactions between the electron and the crystal. The external potential $Q_{\rm ext}$ takes into account the effects of impurities on the otherwise perfect crystal. Here $\hbar$ denotes the Planck constant and $m$ is the mass of the electrons. In many cases of physical interest, the ratio $\eps$ between the mean spacing of the lattice and the characteristic length scale of variation of $Q_{\rm ext}$ is very small. After performing a suitable change of units, and rescaling the external potential and the wave function (see for instance \cite{PR96}) the Schrödinger equation becomes:
\begin{equation}\label{eq:schro}
\left\{
\begin{array}{l}
i\partial_t \psi^\eps (t,x)+\dfrac{1}{2} \Delta_x \psi^\eps(t,x) - \dfrac{1}{\eps^{2}} V_{\rm per}\left(\dfrac{x}{\eps}\right)\psi^\eps(t,x) - V_{\rm ext}(t,x)\psi^\eps(t,x) =0,\vspace{0.2cm} \\
\psi^\eps |_{t=0}=\psi^\eps_0.
\end{array}\right.
\end{equation}
The potential $V_{\rm per}$ is periodic with respect to a fixed lattice in $\R^d$, which, for the sake of definiteness will be assumed to be $\Z^d$.

\medskip

Effective Mass Theory consists in showing that, under suitable assumptions on the initial data $\psi^\eps_0$, the solutions of \eqref{eq:schro} can be approximated for $\eps$ small by those of a simpler Schrödinger equation, the \textit{effective mass equation}, which is of the form:
\begin{equation}\label{eq:embrut}
i\partial_t \phi (t,x)+\dfrac{1}{2} \langle B\, D_x,D_x\rangle\phi(t,x) - V_{\rm ext}(t,x)\phi(t,x)=0,
\end{equation}
where, as usual, $D_x={1\over i}\partial_x$.
The approximation has to be understood in the sense that any weak limit of the density $|\psi^\eps(t,x)|^2dxdt$ is  the density $|\phi(t,x)|^2dxdt$ as $\eps$ goes to~$0$. In equation~(\ref{eq:embrut}), $B$ is a $d\times d$ matrix called the \textit{effective mass tensor}; it generates the {\it effective Hamiltonian}
$$H_{\rm eff}(x,\xi) = {1\over 2}B\xi\cdot \xi + V_{\rm ext}(t,x).$$

The effective mass tensor is an experimentally accessible quantity that can be used to study the effect of the impurities on the dynamics of the electrons. Both the question of finding those initial conditions for which the corresponding solutions of \eqref{eq:schro} converge (in a suitable sense) to solutions to the effective mass equation and that of clarifying the dependence of $B$ on the sequence of initial data have been extensively studied in the literature \cite{BLP78, PR96, AP05, HW11, BBA11}. The effective mass tensor is related to the critical points of the {\it Bloch modes}. These are  the eigenvalues of the operator $P(\xi)$ on $L^2(\T^d)$ which is canonically  associated with the equation~(\ref{eq:schro}),
\begin{equation}\label{def:P}
P(\xi)={1\over 2} |\xi-i\nabla_y|^2 +V_{\rm per}(y),\;\; y\in\T^d,\;\;\xi\in\R^d.
\end{equation}

We focus here on initial data which are structurally related with one of the  Blochs modes in a sense that we will make precise later; we assume that this Bloch mode is of constant multiplicity and we introduce a new method for deriving rigorously equation~(\ref{eq:embrut}). The advantage of this method is that it allows to treat the case where the critical points of the considered Bloch modes are degenerate, leading to the introduction of a new family of Effective mass equations which are of Heisenberg type. 
Our strategy is based on the analysis of the dispersion of PDEs by a Wigner measure approach which has led us to develop global {\it two microlocal Wigner measures} in this specific context, while they are only defined locally in general (\cite{F2micro:00,F2micro:05}).  

\medskip 

Note that different scaling limits  for equation~(\ref{eq:schrodbrut}) have been studied in the literature: 
the interested reader can consult, among many others, references \cite{Ge91,GMMP,PR96, HST:01,BMP01,AP06,CS12,PST:03,DGR06}.

\subsection{Floquet-Bloch decomposition}\label{sec:FB}

The analysis of Schrödinger operators with periodic potentials has a long history that has its origins in the seminal works by Floquet \cite{Floquet1883} on ordinary differential equations with periodic coefficients, and by Bloch \cite{Bloch28}, who developed a spectral theory of periodic Schrödinger operators in the context of solid state physics. Floquet-Bloch theory can be used to study the spectrum of the perturbed periodic Schrödinger operator:
$$-\frac{\eps^2}{2}\Delta_x+V_{\rm per}\left(\frac{x}{\eps}\right)+\eps^2V_{\rm ext}(t,x), $$
see for instance  \cite{RS:V4,Kuch01,Kuch04,Kuch16} and the references therein, and \cite{Out:87,GMS91,HW11} for results in the semiclassical context. The Floquet-Bloch decomposition gives as a result that the corresponding Schrödinger evolution can be decoupled in an infinite family of dispersive-type equations for the so-called \emph{Bloch modes}. We briefly recall the basic facts that we shall need by following the approach in \cite{GMS91,Ge91}. 

\medskip 

The Floquet-Bloch decomposition is based on assuming that the solutions to \eqref{eq:schro} depend on both the ``slow'' $x$ and the ``fast'' $x/\eps$ variables. The fast variables should moreover respect the symmetries of the lattice. This leads to the following Ansatz on the form of the solutions $\psi^\eps$ of~\eqref{eq:schro}:
\begin{equation}\label{def:U}
\psi^\eps(t,x)=U^\eps\left(t,x,\frac{x}{\eps}\right),
\end{equation}
where $U^\eps(t,x,y)$ is assumed to be $\Z^d$-periodic with respect to the variable $y$ (and, therefore, that it can be identified to a function defined on $\R\times\R^d\times\T^d$, where $\T^d$ denotes the torus $\R^d \slash \mathbb{Z}^d$). The function $U^\eps$ then satisfies the  equation:  
\begin{equation}\label{eq:U}
\left\{\begin{array}{l}
i\eps^2\partial_tU^\eps (t,x,y)=  P(\eps D_x) U^\eps (t,x,y) +\eps^2 V_{\rm ext}(t,x) U^\eps(t,x,y),\vspace{0.2cm}\\
U^\eps|_{t=0}= U^\eps_0(x,y), \;\;{\rm such\;\; that}\;\;\psi_{0}^{\eps}=L^\eps U^\eps_0,\end{array}\right.
\end{equation}
where the operator $L^\eps$ maps functions $F$ defined on $\R^d\times \T^d$ on functions on $\R^d$ according to:
\begin{equation}\label{def:Leps}
L^\eps F(x):=F\left(x,\frac{x}{\eps}\right),
\end{equation}
and $P(\eps D_x)$ denotes the operator-valued Fourier multiplier associated with the symbol $\xi\mapsto P(\eps \xi)$ defined in~(\ref{def:P}).
The initial condition in \eqref{eq:U} can be interpreted in terms of the natural embedding $L^2(\R^d_x)\hookrightarrow L^2(\R^d_x\times \T^d_y)$ by taking 
$U^\eps_0(x,y) = \psi^\eps_0(x) \otimes \mathds{1}(y) .$ One can also have more elaborated identifications depending on the structure of the initial data, as we shall see later. 
Identity~\eqref{def:U} makes sense, since one can check that, under suitable assumptions on the initial datum, $U^\eps(t,x,\cdot)$ has enough regularity with respect to the variable $y$ (the fact that $\psi^\eps$ must be given by~\eqref{def:U} following from the uniqueness of solutions to the initial value problem \eqref{eq:schro}).

\medskip 

Assuming that the function $y\mapsto V_{\rm per}(y)$ is smooth is enough for proving that the operator $P(\xi)$  is self-adjoint on $L^2(\T^d)$ (with domain $H^2(\R^d)$) and has a compact resolvent. For the sake of simplicity, we shall make here this assumption, even though it can be relaxed into assuming $V_{\rm per}\in L^p (\T^d)$ for some  convenient set of indices~$p$ which authorizes Coulombian singularity in dimension~$3$ (see~\cite{lewin_lecture}).
As a consequence of the fact that $P(\xi)$ has compact resolvent,  there exist a non-decreasing sequence of eigenvalues (the so-called \textit{Bloch energies}):
$$\varrho_1(\xi)\leq \varrho_2(\xi)\leq \cdots\leq \varrho_n(\xi) \leq \cdots \longrightarrow +\infty,$$
and an orthonormal basis of $L^2(\T^d)$ consisting of eigenfunctions $\left(\varphi_n(\xi,\cdot)\right)_{n\in\N}$ (called \textit{Bloch waves}): 
$$P(\xi)\varphi_n(y,\xi)=\varrho_n(\xi)\varphi_n(y,\xi),\quad \text{ for }y\in\T^d.$$
Moreover, the Bloch energies $\varrho_n(\xi)$ are $2\pi\Z^d$-periodic whereas the Bloch waves satisfy $$\varphi_n(y,\xi+2\pi k)=e^{-i2\pi k\cdot y}\varphi_n(y,\xi),\quad \text{ for every } k\in\Z^d.$$ 
This follows from the fact that for every $k\in\Z^d$, the operator $P(\xi+2\pi k)$ is unitarily equivalent to $P(\xi)$ since $P(\xi + 2\pi k) = {\rm e}^{-i2\pi k\cdot y}P(\xi){\rm e}^{i2\pi k\cdot y}$.
It is proved in \cite{Wilcox78} that the Bloch energies $\varrho_n$ are continuous and piecewise analytic functions of $\xi\in\R^d$. Actually, the set~$\{(\xi,\varrho_n(\xi)),\;n\in\N,\;\;\xi\in\R^d \}$ is an analytic set of $\R^{2d}$. Moreover, if the multiplicity of the eigenvalue~$\varrho_n(\xi)$ is equal to the same constant for all $\xi\in\R^d$, then $\varrho_n$  and the eigenprojector~$\Pi_n$ on this mode are  globally analytic functions of $\xi$. The reader can refer to~\cite{Kuch16} for a survey on the subject.

\medskip 

Observing that, via the decomposition in Fourier series, any function $U\in L^2(\R^d_x\times\T^d_y)$ can be written as:
\[
U(x,y)=\sum_{k\in\Z^d}U_k(x){\rm e}^{i2\pi k\cdot y}\;\;{\rm 
with}\;\;
\|U\|_{L^2(\R^d\times\T^d)}^2=\sum_{k\in\Z^d}\|U_k\|_{L^2(\R^d)}^2,
\]
we denote by $H^s_\eps(\R^d\times\T^d)$, for $s\geq 0$, the Sobolev space consisting of those functions $U \in L^2(\R^d\times\T^d)$ such that there exists $C>0$ and: 
\begin{equation}\label{def:Hseps}
\forall \eps>0,\;\;\|U\|_{H^s_\eps(\R^d\times\T^d)}^2:=\sum_{k\in\Z^d}\int_{\R^d}(1+|\eps\xi|^2+|k|^2)^s|\widehat{U_k}(\xi)|^2d\xi\leq C,
\end{equation}
where 
$\displaystyle{\widehat{U_k}(\xi)= \int_{\R^d} {\rm e}^{-ix\cdot \xi} U_k(x) dx.}$

\subsection{Main result}
We consider the following set of assumptions. 

\begin{assumption} \label{hypothesis:main}
\begin{enumerate}
\item Assume $V_{\rm per}$ is smooth and real-valued
and that  $V_{\rm ext}$ is a continuous function in time taking values in the set of smooth, real-valued,  bounded functions on~$\R^d$ with bounded derivatives.
\item  Assume that $\varrho_n$ is  a Bloch mode 
 of constant multiplicity and that the set of critical points of $\varrho_n$ 
  $$\Lambda_n:=\{\xi\in\R^d,\;\; \nabla \varrho_n(\xi)=0\}$$
  is a submanifold  of $\R^d$.
  \item Assume that  the Hessian $d^2 \varrho_n(\xi)$ is of maximal rank above each point $\xi\in\Lambda_n$  (or equivalently that  ${\rm Ker} \, d^2 \varrho_n(\xi)=  T_\xi \Lambda_n$ for all $\xi\in\Lambda_n$).
 \item Assume that the initial data $\psi^\eps_0(x)$ satisfies 
 $$
 \psi^\eps_0(x)= U^\eps_0\left(x,{x\over \eps}\right)\;\; {\rm with} \;\; \widehat U^\eps_0(\xi, \cdot)\in {\rm Ran} \,\Pi_n(\eps \xi),$$
with  $U^\eps_0$  uniformly bounded in $H^s_\eps(\R^d\times\T^d)$ for some $s>d/2$. 
\end{enumerate}
 \end{assumption}

It will be convenient to identify $\varrho_n$ to a function defined on $(\R^d)^*$ rather than $\R^d$ (via the standard identification given by duality). Then we define the cotangent bundle of $\Lambda_n$ as the union of all cotangent spaces to $\Lambda_n$
\begin{equation}\label{def:T*X}
T^*\Lambda_n:=\{(x,\xi)\in \R^d\times \Lambda_n \, : \, x \in T_\xi^* \Lambda_n \},
\end{equation}
each fibre $T_\xi^*\Lambda_n$ is the dual space of the tangent space $T_\xi\Lambda_n$. Note that this is well-defined, since $T_\xi^*\Lambda_n\subset (\R^d)^{**}=\R^d$. We 
shall denote by $\mathcal{M}_+(T^*\Lambda_n)$ the set of positive Radon measures on~$T^*\Lambda_n$. We
also define the normal bundle of $\Lambda_n$ which is the union of those linear subspaces of~$\R^d$ that are normal to $\Lambda_n$:
\begin{equation}\label{def:NX}
N\Lambda_n:=\{(z,\xi)\in \R^d\times \Lambda_n \, : \, z \in N_\xi \Lambda_n \},
\end{equation}
 where $N_\xi \Lambda_n$ consists of those $x\in(\R^d)^{**}=\R^d$ that annihilate $T_\xi \Lambda_n$.
Every point $x\in\R^d$ can be uniquely written as $x=v+z$, where $v\in T^*_\xi\Lambda_n$ and $z\in N_\xi\Lambda_n$. Given a function $\phi\in L^\infty(\R^d)$ we write $m_\phi(v,\xi)$, where $v\in T^*_\xi\Lambda_n$, to denote the operator acting on $L^2(N_\xi\Lambda_n)$ by multiplication by $\phi(v+\cdot)$.   Note that assumption (3) implies that the Hessian of $\varrho_n$ defines an operator $d^2\varrho_n(\xi)D_z\cdot D_z$ acting on $L^2(N_\xi\Lambda_n)$ for any $\xi\in\Lambda_n$. 

\medskip 

In the statement below, 
the weak limit of the energy density are described by means of 
a time-dependent family $M_n$ of trace-class operators acting on a certain $L^2$-space. More precisely, the operators $M_n$ depend on $t\in\R$ and on $\xi\in\Lambda_n$, $v\in T^*_\xi\Lambda_n$; for every choice of these parameters, $M_n(t,v,\xi)$ is a trace-class operator acting on  $L^2$ functions of the vector space~$N_\xi\Lambda$.  Note that $M_n(t,\cdot)$ can also be viewed as a section of a vector bundle over $T^*\Lambda_n$, namely: $\bigsqcup_{(v,\xi)\in T^*\Lambda}{\mathcal L}^1_+\left(L^2(N_\xi\Lambda_n)\right)$.

\begin{theorem}\label{mainresult}
Assume the hypotheses of Assumption~\ref{hypothesis:main}.  Then, 
there exist a subsequence $(\eps_k)_{k\in\N}$, a positive measure $\nu_n\in\mathcal{M}_+(T^*\Lambda_n)$, and a measurable family of self-adjoint, positive, trace-class operator 
$$M_{0,n}:T^*_\xi\Lambda_n\ni (v,\xi)\longmapsto M_{0,n}(v,\xi)\in \mathcal{L}_+^1(L^2(N_\xi\Lambda_n)),\quad \Tr_{L^2(N_\xi\Lambda_n)} M_{0,n}(v,\xi)=1,$$ 
such that for every 
for every $a<b$ and every $\phi\in \mathcal{C}_c(\R^d)$ one has:
$$\lim_{k\to\infty}\int_a^b\int_{\R^d}\phi(x)|\psi^{\eps_k}(t,x)|^2dxdt
=\int_a^b \int_{T^*\Lambda_n}\Tr_{L^2(N_\xi\Lambda_n)}\left[m_\phi(v,\xi)M_n(t,v,\xi)\right]\nu_n(dv,d\xi)dt,
$$
where $M_n(\cdot, v,\xi)\in \mathcal{C}(\R;\mathcal{L}_+^1(L^2(N_\xi\Lambda_n))$ solves the Heisenberg equation:
\begin{equation}\label{eq:heis}
\left\{ \begin{array}{l}
i\partial_t M_n(t,v,\xi) +\left[\dfrac{1}{2} d^2\varrho_n(\xi)D_z\cdot D_z + m_{V_{\rm ext}(t,\,\cdot\,)}(v,\xi), M_n(t,v,\xi)\right] =0,\vspace{0.2cm}\\
M_n |_{t=0}=M_{0,n}.
\end{array} \right.
\end{equation}
\end{theorem}

\begin{remark}
We point out that the measure $\nu_n$ and the family of operators $M_{0,n}$ only depend on the subsequence $\psi^{\eps_k}_0$ of initial data. The way of computing them  will be made clear in Section~\ref{sec:4}.
\end{remark}

When the critical points of $\varrho_n(\xi)$ are all non degenerate, then $\Lambda_n$ is discrete and $2\pi\Z^d$-periodic, $T^*\Lambda_n=\Lambda_n\times \{0\}$  and  $N\Lambda_n= \R^d$. We then have the following corollary.

\begin{corollary}\label{mainresul:case1}
Assume we have Assumption~\ref{hypothesis:main} and that the critical points of $\varrho_n(\xi)$ are all non degenerate.
Then the measure $\nu_{n}$ and the operator $M_{n}$ of Theorem~\ref{mainresult} above satisfy:
\begin{enumerate}
\item The operator $M_{n}(t,\xi)$ is the orthogonal projection on
$\psi_\xi$ which  solves the effective mass equation:
\begin{equation}\label{eq:schroh}
i\partial_t \psi_\xi(t,x) ={1\over 2}d^2\varrho_{n}(\xi)D_x\cdot D_x \psi_\xi(t,x) +V_{\rm ext}(t,x)\psi_\xi(t,x),
\end{equation}
with initial data:
\begin{center}
$\psi_\xi|_{t=0}$ is the weak limit in $L^2(\R^d)$ of the sequence $\left({\rm e}^{-\frac{i}{\eps_k} \xi \cdot x} \psi^{\eps_k}_{0}\right)$.
\end{center}
\item The measure $\nu_{n}$ is given by 
$$\nu_{n}=\sum_{\xi\in \Lambda_{n}}\alpha_\xi \delta_\xi,\;\; \alpha_\xi = \| \psi_\xi|_{t=0}\|_{L^2(\R^d)}.$$
\end{enumerate}
\end{corollary}

This corollary is  well known and we refer to  the work by Allaire and Piatnitski~\cite{AP05} or to  \cite{AP06} for similar results in a related problem; in that work homogenization and two-scale convergence techniques are used to obtain a precise description of the solution profile for data similar to ours and for Blochs mode having non-degenerated critical points. In~\cite{BBA11}, Barletti and Ben~Abdallah obtained a result similar to Corollary \ref{mainresul:case1} by following the approach initiated by Kohn and Luttinger in~\cite{LuttingerKohn} consisting in introducing a (non-canonical) basis of modified Bloch functions. 

\medskip 

The starting point in our approach is conceptually closer to that in~\cite{PR96}, in the sense that we analyse the structure of Wigner measures associated to sequences of solutions. The main novelty here is the use of two-microlocal Wigner measures, that give a more explicit geometric description of the mechanism that underlies the Effective Mass Approximation, showing that it is a result of the dispersive effects associated to high-frequency solutions to the semiclassical Bloch band equations. Moreover, we are able to deal with the presence of non-isolated critical points on the Bloch energies and to prove Theorem~\ref{mainresult}. We believe our approach is sufficiently robust to be implemented on a Bloch band, isolated from the remainder of the spectrum,  and consisting of several Bloch modes which may present crossings. We will devote further works to this specific problem. 
It is also interesting to notice that our result generalizes to initial data which are a finite sum of data satisfying (4) of Assumption~\ref{hypothesis:main}. The weak limit of the energy density associated with the solution corresponding to this new data is the sum of weak limits of the energy densities of the solution associated with each term of the data, without any interference (see section~\ref{sec:superposition} for a precise statement).

\subsection{Strategy of the proof}

The proof of Theorem~\ref{mainresult}  relies on the analysis of the solution~$U^\eps$  to equation~(\ref{eq:U}) with initial data $U^\eps_0$ as introduced in (4) of Assumption~\ref{hypothesis:main}, and more precisely on its  component $U^\eps_n$ on the $n$-th Bloch mode  and its restriction $\psi^\eps_n$ by $L^\eps$:
 $$U^\eps_n= \Pi_n(\eps D_x)U^\eps,\;\;
\psi^\eps_n = L^\eps U^\eps_n.$$
It is shown in Section~\ref{sec:psiepsn} that the family $(\psi^\eps_n)$ solves the equation 
\begin{equation}\label{eq:U_components}
\left\{ \begin{array}{l}
i\eps^2 \partial_{t} \psi^\eps_{n} (t,x)- \varrho_{n}(\eps D_x) \psi^\eps_{n}(t,x)- \eps^2 V_{\rm ext}(t,x) \psi^\eps_{n}(t,x)= \eps^2 f^\eps_n(t,x),  \vspace{0.2cm}\\
\psi^\eps_{n}|_{t=0}(x)= \psi_{0}^{\eps}(x)
\end{array}\right.
\end{equation}
 with 
  $$f^\eps_n=L^\eps  \left[  \Pi(\eps D_x) , V_{\rm ext}\right] U^\eps,$$
There, we prove that   
\begin{equation}\label{eq:adiabatic}
\forall T\in\R,\;\;\exists C_T>0,\;\; \sup_{t\in[0,T]} \| \psi^\eps(t,\cdot\,)- \psi^\eps_n(t,\cdot\,)\|_{L^2(\R^d)}\leq C_T \eps .
\end{equation}
and 
\begin{equation}\label{eq:fneps}
\exists C>0,\;\; 
\forall t\in\R,\;\;\|f_n^\eps(t,\cdot\,)\|_{L^2(\R^d)}\leq C\eps. 
\end{equation}
Equation~(\ref{eq:adiabatic})   shows that no other Bloch modes is concerned in the decomposition of $U^\eps$ and~$\psi^\eps$: the mass of $\psi^\eps$ remains above the specific mode $\varrho_n$ because it is separated from the other ones. 
Therefore, a crucial step in this strategy  consists in performing a detailed analysis of the dispersive equation~(\ref{eq:U_components}).

\subsection{Structure of the article}
Sections~\ref{sec:disp}
 to~\ref{sec:4}  are devoted to the analysis of a dispersive equation of the form~(\ref{eq:U_components}) in a more general setting. For this, we use pseudodifferential operators and semi-classical measures (Section~\ref{sec:semiclassical}) and we introduce two-microlocal tools (Section~\ref{sec:two_microlocal}) that allow us to prove the main results of Section~\ref{sec:disp}   in Section~\ref{sec:4}. Finally, in Section~\ref{sec:5}  we come back to the effective mass equations and prove Theorem~\ref{mainresult} , which requires additional results on the restriction operator~$L^\eps$, the projector~$\Pi_n(\xi)$ and energy estimates for solutions to \eqref{eq:U}. Some Appendices are devoted to basic results about pseudodifferential calculus and trace-class operator-valued measures, and to the proof of technical lemmata.


\section{Quantifying the lack of dispersion}\label{sec:disp}

As emphasized in the introduction,  understanding the limiting behavior as $\eps\to0$ of the position densities of solutions to the Schrödinger equation \eqref{eq:schro} relies on a careful analysis of the solutions of equations of the form: 
\begin{equation}\label{eq:disp2}
\left\{\begin{array}{l}
i\eps^2 \partial_{t} u^{\eps}(t,x) = \lambda(\eps D_x) u^{\eps}(t,x) + \eps^2 V_{\rm ext}(t,x) u^{\eps}(t,x)+\eps^3g^\eps(t,x) , \; (t,x)\in\R\times\R^d,\\
u^\eps_{|t=0}=u^\eps_0,
\end{array}
\right.
\end{equation}
where  $(g^\eps(t,\cdot))$ is locally uniformly bounded with respect to $t$ in $L^2(\R^d)$. 

\medskip

This equation ceases to be dispersive as soon as $\lambda(\xi)$ has critical points $\xi \neq 0$, and this is always the case if $\lambda$ is a Bloch energy. Heuristically, one can think that one of the consequences of a dispersive time evolution is a regularization of the high-frequency effects (that is associated to frequencies $\eps\xi=c\neq 0$) caused by the sequence of initial data. These heuristics have been made precise in many cases; a presentation of our results from this point of view can be found in \cite{CFMProc}. The reader can also find there a detailed account on the literature on the subject.

\medskip 

Here we show that, in the presence of critical points of $\lambda$, some of the high-frequency effects exhibited by the sequence of initial data persist after applying the time evolution~\eqref{eq:disp2}.  We provide a quantitative picture of this persistence by giving a complete description of the asymptotic behavior of the densities $|u^\eps(t,x)|^2dxdt$ associated to a bounded sequence $(u^\eps)$ of solutions to \eqref{eq:disp2}. We  give an explicit procedure to compute all weak-$\star$ accumulation points of the sequence of positive measures $(|u^\eps(t,x)|^2dxdt)$ in terms of quantities that can be obtained from the sequence of initial data $(u_{0}^{\eps})$. These results are of independent interest; we have thus chosen to present them in a more general framework than what it is necessary in our applications to Effective Mass Theory. 

\medskip 

In order to obtain a non trivial result we must make sure that the characteristic length-scale of the oscillations carried by the sequence of initial data is of the order of $\eps$. The following assumption is sufficient for our purposes:

\begin{enumerate}
\item[\textbf{H0}] The sequence $(u^\eps_0)$ is uniformly bounded in $L^2(\R^d)$ and $\eps$-oscillating, in the sense that its energy is concentrated on frequencies smaller or equal than $1/\eps$ :
\begin{equation}\label{def:eps-osc}
\limsup_{\eps\rightarrow 0^+}\int_{|\xi|>R/\eps} | \widehat {u^\eps_0} (\xi)| ^2 d\xi \Tend{R}{+\infty} 0.
\end{equation}
\end{enumerate} 

We shall assume that $\lambda$ is smooth and grows at most polynomially, and that its set of critical points is a submanifold of $\R^d$.  More precisely, we impose the following hypotheses on $\lambda$ and $V$:

\begin{enumerate}
\item[\textbf{H1}] $V_{\rm ext}\in C^\infty(\R\times\R^d)$ is bounded together with its derivatives and $\lambda\in C^\infty(\R^d)$ , together with its derivatives, grows at most polynomially; \textit{i.e.} there exists $N>0$ such that, for every $\alpha\in\N^d_+$, one has:
$$\sup_{\xi\in\R^d}|\partial_\xi^\alpha\lambda(\xi)|(1+|\xi|^N)^{-1}<\infty.$$
\item[\textbf{H2}] The set $$\Lambda := \left\lbrace \xi \in \R^d : \nabla\lambda(\xi) = 0 \right\rbrace$$ is a connected, closed embedded submanifold of $\R^d$ of codimension $0<p\leq d$ and the Hessian $d^2\lambda$ is of maximal rank over $\Lambda$. 

\end{enumerate}

Hypothesis {\bf H2} implies the existence of tubular coordinates in a neighborhood of~$\Lambda$. A stronger version of {\bf H2} is to suppose that all critical points of $\lambda$ are non-degenerate (that is, the Hessian of $\lambda$, $d^2\lambda(\xi)$ is a non-degenerate quadratic form for every $\xi\in\Lambda$). This implies that $p=d$ and $\Lambda$ is a discrete set in~$\R^d$; if moreover one has that $\lambda$ is $\Z^d$-periodic, which is the situation when $\lambda$ is a Bloch energy, this set is finite modulo $\Z^d$. We first state the main result of this section under this stronger hypothesis. 

\begin{theorem}\label{theo:disc}
Suppose that the sequence of initial data $(u^\eps_0)$ verifies {\bf H0}, denote by~$(u^\eps)$ the corresponding sequence of solutions to \eqref{eq:disp2}. Suppose in addition that {\bf H1} is satisfied and all critical points of $\lambda$ are non-degenerate. 
Then there exists a subsequence $(u^{\eps_k}_0)$ such that for every $a<b$ and every $\phi\in \mathcal{C}_c(\R^d)$ the following holds:
\begin{equation}\label{eq:limu}
\lim_{k\to\infty}\int_a^b\int_{\R^d}\phi(x)|u^{\eps_k}(t,x)|^2dxdt=\sum_{\xi\in \Lambda}\int_a^b\int_{\R^d}\phi(x)|u_\xi (t,x)|^2dxdt,
\end{equation}
where $u_\xi$ solves the following Schrödinger equation:
\begin{equation}\label{eq:schrohprofil}
i\partial_t u_\xi(t,x) =d^2\lambda(\xi)D_x\cdot D_x u_\xi(t,x) +V_{\rm ext}(t,x)u_\xi(t,x),
\end{equation}
with initial data:
\begin{center}
$u_\xi|_{t=0}$ is the weak limit in $L^2(\R^d)$ of the sequence $\left({\rm e}^{-\frac{i}{\eps_k} \xi \cdot x} u^{\eps_k}_0\right)$. 
\end{center}
If $\Lambda=\emptyset$ then the right-hand side of \eqref{eq:limu} is equal to zero.
\end{theorem}

Note that $u_\xi$ may be identically equal to zero even if the sequence $(u^\eps_0)$ oscillates in the direction $\xi$. For instance, if the sequence of initial data is a coherent state:
$$u^\eps_0(x)=\frac{1}{\eps^{d/4}}\rho\left(\frac{x-x_0}{\sqrt{\eps}}\right){\rm e}^{\frac{i}{\eps}\xi_0 \cdot x},$$
centered at a point $(x_0,\xi_0)$ in phase space with $\rho \in C_0(\R^d)$, then $u_\xi|_{t=0}=0$ for every $\xi\in\R^d$. Theorem \eqref{theo:disc} allows us to conclude that the corresponding solutions $(u^\eps)$ converge to zero in $L^2_{\loc}(\R\times\R^d)$. 

\medskip 

Theorem \ref{theo:disc} can be interpreted as a description of the obstructions to the validity of smoothing-type estimates for the solutions to equation \eqref{eq:disp2} in the presence of critical points of the symbol of the Fourier multiplier. We refer the reader to \cite{CFMProc} for additional details concerning this issue and a simple proof of Theorem~\ref{theo:disc}. Here, 
we obtain Theorem \ref{theo:disc} as a particular case of a more general result which requires some geometric preliminaries. 

\medskip 

As for the mode Bloch $\varrho_n$ in the Introduction, we identify $\lambda$ to a function defined on $(\R^d)^*$ rather than $ \R^d$, and we associate with $\Lambda$ its 
 cotangent bundle  $T^*\Lambda$ and its normal bundle  $N\Lambda$.
In the analogue of Theorem \ref{theo:disc} in this context, the sum over critical points is replaced by an integral with respect to a measure over $T^*\Lambda$, and the Schr\"odinger equation \eqref{eq:schrohprofil} becomes a Heisenberg equation for a time-dependent family $M$ of trace-class operators 
of $\bigsqcup_{(v,\xi)\in T^*\Lambda}{\mathcal L}^1_+\left(L^2(N_\xi\Lambda)\right)$.

\begin{theorem}\label{theo:nondis}
Let $(u^\eps_0)$ be a sequence of initial data satisfying {\bf H0}, and denote by $(u^\eps)$ the corresponding sequence of solutions to \eqref{eq:disp2}. If {\bf H1} and {\bf H2} hold, then there exist a subsequence $(u^{\eps_k}_0)$, a positive measure $\nu\in\mathcal{M}_+(T^*\Lambda)$ and a measurable family of self-adjoint, positive, trace-class operators 
$$M_0:T^*_\xi\Lambda \ni (v,\xi)\longmapsto M_0(z,\xi)\in \mathcal{L}_+^1(L^2(N_\xi\Lambda)),\quad \Tr_{L^2(N_\xi\Lambda)} M_0(v,\xi)=1,$$ 
such that for every $a<b$ and every $\phi\in \mathcal{C}_c(\R^d)$ one has:
\begin{equation}\label{eq:limubis}
\lim_{k\to\infty}\int_a^b\int_{\R^d}\phi(x)|u^{\eps_k}(t,x)|^2dxdt=\int_a^b \int_{T^*\Lambda}\Tr_{L^2(N_\xi\Lambda)}\left[m_\phi(v,\xi)M_t(v,\xi)\right]\nu(dv,d\xi)dt,
\end{equation}
where $t\mapsto M_t(v,\xi)\in \mathcal{C}(\R;\mathcal{L}_+^1(L^2(N_\xi\Lambda))$ solves the following Heisenberg equation:
\begin{equation}\label{eq:heis1}
\left\{ \begin{array}{l}
i\partial_t M_t(v,\xi) =\left[\dfrac{1}{2} d^2\lambda(\xi)D_z\cdot D_z + m_{V_{\rm ext}(t,\cdot)}(v,\xi), M_t(v,\xi)\right] ,\vspace{0.2cm}\\
M |_{t=0}=M_0.
\end{array} \right.
\end{equation}
\end{theorem}

\begin{remark}
When hypothesis {\bf H2} about the rank of the Hessian $d^2\lambda$ is dropped, then an additional term appears in~(\ref{eq:limubis}) (see~\cite{CFMProc}).
\end{remark}

When $\Lambda$ consists of a set of isolated critical points, Theorems \ref{theo:disc} and \ref{theo:nondis} are completely equivalent. Note that in this case, $T^*\Lambda=\{0\}\times\Lambda$ and the measure $\nu$ (which in this case is a measure depending on $\xi\in\R^d$ only) is simply 
$$\nu=\sum_{\xi\in\Lambda}\alpha_\xi\delta_\xi,$$
where $\alpha_\xi = \|u_\xi|_{t=0}\|^2_{L^2(\R^d)}$. In addition, $N_\xi\Lambda=\R^d$ and the operator $M_t(\xi)$ (which again does not depend on $z$) is the orthogonal projection onto $u_\xi(t,\cdot)$ in $L^2(\R^d)$ (recall that $u_\xi$ solves the Schrödinger equation \eqref{eq:schroh}). These orthogonal projections satisfy the Heisenberg equation \eqref{eq:heis1}.

\medskip

The proof of Theorem~\ref{theo:nondis} follows a strategy developed in the references~\cite{MaciaTorus,AM:12,AFM:15} in a different (though related) context.  
As in those references, the measure~$\nu$ and the family of operators $M_0$
only depend on the subsequence of initial data $(u^{\eps_k}_0)$; we will see in Section~\ref{sec:semiclassical} that they are defined as two microlocal Wigner measures of $(u^{\eps_k}_0)$ in the sense of~\cite{Fethese95,F2micro:00,F2micro:05,MaciaTorus}. At this point, it might be useful to stress out that in this regime the limiting objects $M,\nu$ cannot be computed in terms of the Wigner/semiclassical measure of the sequence of initial data, as it is the case when dealing with the semiclassical limit. 
 In~\cite{CFMProc},  we have explicitly constructed sequences of initial data having the same semiclassical measure but such that their time dependent measures differ.
 This type of behavior was first remarked in this context in the case of the Schrödinger equation on the torus, see \cite{MaciaAv, MaciaTorus}.

\medskip

We also emphasize that the original definition of  two-microlocal Wigner measures performed in~\cite{F2micro:00} and their extension to more general geometric setting~\cite{F2micro:05} were only defined locally. We prove here that they extend to global objects in the geometric context of closed simply connected embedded submanifolds of~$\R^d$; related constructions were performed in   the torus~\cite{MaciaTorus,AM:12,AFM:15,MR:18} and the disk~\cite{ALM:16}.

\medskip

See also, that as soon as $\Lambda$ has strictly positive dimension (\textit{i.e.} it is not a union of isolated critical points), the measure $\nu$ may be singular with respect to the $z$ variable, while when $\Lambda$ consists in isolated points, the weak limit of the densities $|\psi^\eps(t,x)|^2dx$ are proved to be absolutely continuous with respect to the measure $dx$. See \cite{CFMProc} for specific examples exhibiting this type of behavior; see also that reference for examples proving the necessity of hypothesis {\bf H2}; it is shown there that different types of behavior can happen whenever the Hessian of $\lambda$ is not of full rank on $\Lambda$. 

\medskip 

The main idea of the proof comes from the following remark. Setting 
$v^\eps(t,x)= u^\eps(\eps t, x),$
then $(v^\eps)$ solves the semi-classical equation 
\begin{equation}\label{eq:semischro}
\left\{\begin{array}{l}
i\eps \partial_{t} v^{\eps}(t,x) = \lambda(\eps D_x) v^{\eps}(t,x) + \eps^2 V_{\rm ext}(t,x) v^{\eps}(t,x)+\eps^3g^\eps(t,x) , \; (t,x)\in\R\times\R^d,\\
v^\eps_{|t=0}=u^\eps_0,
\end{array}
\right.
\end{equation}
which means that, in the preceding analysis, we have performed 
the semiclassical limit $\eps\to 0$ in \eqref{eq:semischro} simultaneously with the limit $t / \eps\to +\infty$.  Such analysis, combining high-frequencies ($\eps\to 0$) and long times ($t\sim t_\eps\to +\infty$) is relevant if one wants to understand the behavior of solutions of \eqref{eq:semischro} beyond the Ehrenfest time. This approach was followed in the case of confined geometries in the references \cite{MaciaAv, AFM:15, MR:16}. Note also that in the particular case when $\lambda(\xi)$ is homogeneous of degree two, this change of time scale transforms the semiclassical equation~\eqref{eq:semischro} into the non-semiclassical one (that is, the one corresponding to $\eps=1$). Therefore, it is possible to derive  results on the dynamics of the Schrödinger equation via this scaling limit \cite{MaciaTorus, AnRiv, AM:14, ALM:16}. The reader can consult the survey articles \cite{MaciaDispersion, AM:12} and the introductory lecture notes \cite{MaciaLille} for additional details and references on this approach.

\section{Pseudodifferential operators and semiclassical measures -- preliminaries}\label{sec:semiclassical}

In this section we recall some basic facts on Wigner distributions and semiclassical measures, which are the tools we are going to use to prove Theorem \ref{theo:nondis} and derive preliminary results about Wigner measures associated with families of solutions of equations of the form~(\ref{eq:disp2}). 

\subsection{Wigner transform and Wigner measures}
Wigner distributions provide a useful way for 
computing  weak-$\star$ accumulation points of a sequence of densities $|f^\eps(x)|^2dx$ constructed from a $L^2$-bounded sequence $(f^\eps)$ of solutions of a semiclassical (pseudo) differential equation. They provide a joint physical/Fourier space description of the energy distribution of functions in $\R^d$.
The Wigner distribution of a function $f\in L^2(\R^d)$ is defined as:
$$W^\eps_f(x,\xi):=\int_{\R^d}f\left(x-\frac{\eps v}{2}\right)\overline{f\left(x+\frac{\eps v}{2}\right)}{\rm e}^{i\xi\cdot v}\frac{dv}{(2\pi)^d}, $$
and has several interesting properties (see, for instance, \cite{FollandPhaseSpace}). 
\begin{itemize}
\item $W^\eps_f\in L^2(\R^d\times\R^d)$.
\item Projecting $W_f^\eps$ on $x$ or $\xi$  gives the position or momentum densities of $f$ respectively :
$$\int_{\R^d}W_f^\eps(x,\xi)d\xi=|f(x)|^2,\quad \int_{\R^d}W_f^\eps(x,\xi)dx=\frac{1}{(2\pi\eps)^d}\left|\widehat{f}\left(\frac{\xi}{\eps}\right)\right|^2.$$
Note that despite this, $W_f^\eps$ is not positive in general.
\item For every $a\in\mathcal{C}^\infty_c(\R^d\times\R^d)$ one has:
\begin{equation}\label{eq:wbdd}
\int_{\R^d\times\R^d}a(x,\xi)W_f^\eps(x,\xi)dx\,d\xi=(\op_\eps(a)f,f)_{L^2(\R^d)},
\end{equation}
where $\op_\eps(a)$ is the semiclassical pseudodifferential operator of symbol $a$ obtained through the Weyl quantization rule: 
$$\op_\eps(a) f(x)=\int_{\R^{d}\times\R^d} a\left(\frac{x+y}{2},\eps\xi\right) {\rm e}^{i \xi\cdot (x-y)} f(y)dy\,\frac{d\xi}{(2\pi)^d}.$$  
\end{itemize} 

If $(f^\eps)$ is a bounded sequence in $L^2(\R^d)$ then $(W^\eps_{f^\eps})$ is a bounded sequence of tempered distributions in $\mathcal{S}'(\R^d\times\R^d)$. This is proved using identity \eqref{eq:wbdd} combined with the fact that the operators $\op_\eps(a)$ are uniformly bounded by a suitable semi-norm in $\mathcal{S}(\R^d\times\R^d)$, see \eqref{est:pseudo}. Appendix \ref{sec:pdo} contains additional facts on the theory of pseudodifferential operators, as well as references to the literature.

\medskip 

In addition, every accumulation point of $(W^\eps_{f^\eps})$ in $\mathcal{S}'(\R^d\times\R^d)$ is a positive distribution and therefore, by Schwartz's theorem, a positive measure on $\R^d\times\R^d$. These measures are called \textit{semiclassical} or \textit{Wigner measures}. See references \cite{Ge91,LionsPaul,GerLeich93,GMMP} for different proofs of the results we have presented so far.

\medskip 

Now, if $\mu\in\mathcal{M}_+(\R^d\times\R^d)$ is an accumulation point of $(W^\eps_{f^\eps})$ along some subsequence $(\eps_k)$
and $(|f^{\eps_k}|^2)$ converges weakly-$\star$ towards a measure $\nu\in\mathcal{M}_+(\R^d)$ then one has:
\begin{equation}\label{muleqnu}
\int_{\R^d}\mu(\cdot,d\xi)\leq \nu.
\end{equation}
Equality holds if and only if $(f^\eps)$ is $\eps$-oscillating:
\begin{equation}\label{def:xoscclassical}
\limsup_{\eps\rightarrow 0^+}\int_{|\xi|>R/\eps} | \widehat {f^\eps} (\xi)| ^2 d\xi \Tend{R}{+\infty} 0,
\end{equation} 
see \cite{Ge91,GerLeich93,GMMP}. The hypothesis {\bf H0} that we made on the initial data for equation \eqref{eq:disp2}, is this $\eps$-oscillating property.
Note also that~(\ref{muleqnu}) implies that $\mu$ is always a finite measure of total mass bounded by $\sup_\eps \|f^\eps\|^2_{L^2(\R^d)}$.

\begin{remark}\label{rem:Hsepsosc}
If $\|\langle \eps D_x\rangle^s f^\eps\|_{L^2(\R^d)}$ is uniformly bounded for some constant $s>0$, then the family $f^\eps$ is $\eps$-oscillating.
\end{remark}

\subsection{Wigner measure and family of solutions of dispersive equations}

We will now consider Wigner distributions associated to solutions of the evolution equation~(\ref{eq:disp2}) where $V_{\rm ext}$ and $\lambda$ satisfy hypothesis {\bf H1} and $(g^\eps(t,\cdot))$ is locally uniformly bounded with respect to $t$ in $L^2(\R^d)$.

\medskip

When the sequence $(u^\eps_0)$ of initial data is uniformly bounded in $L^2(\R^d)$, so is the corresponding sequence $(u^\eps(t,\cdot))$ of solutions to \eqref{eq:disp2} for every $t\in\R$. Therefore the sequence of Wigner distributions $(W^\eps_{u^\eps(t,\cdot)})$ is bounded in $\mathcal{C}(\R;\mathcal{S}'(\R^d\times\R^d))$. Nevertheless, its time derivatives are unbounded and, in general, one cannot hope to find a subsequence that converges pointwise (or even almost everywhere) in~$t$ (see Proposition \ref{prop:loc} below).
This difficulty can be overcome if one considers  the time-average  of the Wigner distributions. 

\begin{proposition}\label{prop:msc}
Let $(u^{\eps})$ be a sequence of solutions to \eqref{eq:disp2} issued from an $L^2(\R^d)$-bounded family of initial data $(u^\eps_0)$. Then there exist a subsequence $(\eps_k)$ tending to zero as $k\to\infty$ and a $t$-measurable family $\mu_t\in\mathcal{M}_+(\R^d\times\R^d)$ of finite measures, with total mass essentially uniformly bounded in $t\in\R$, such that, for every $\theta\in L^1(\R)$ and $a\in\mathcal{C}^\infty_c(\R^d\times\R^d)$:
$$\lim_{k\to\infty}\int_{\R\times\R^d\times\R^d} \theta(t) a(x,\xi)W_{u^{\eps_k}(t,\cdot)}^{\eps_k}(x,\xi)dx\,d\xi\, dt=\int_{\R\times\R^d\times\R^d}\theta(t)a(x,\xi)\mu_t(dx,d\xi)dt.$$
If moreover,  the families $(u^\eps_0)$ and $g^\eps(t,\cdot)$ are $\eps$-oscillating, then for every $\theta\in L^1(\R)$ and $\phi\in {\mathcal C}_c^\infty(\R^d)$: 
$$\lim_{k\to\infty}\int_\R\int_{\R^d} \theta(t) \phi(x)|u^{\eps_{k}}(t,x)|^2 dx \,dt=\int_\R\int_{\R^d\times\R^d} \theta(t)\phi(x) \mu_t( dx,d\xi) dt.$$ 
\end{proposition}

This result is proved in \cite{MaciaAv}, Theorem 1; see also Appendix B in \cite{MR:16}. 
Note that its proof uses the following observation.

\begin{remark} \label{remark:tx} Let $(u^{\eps}(t,\cdot))$ be a sequence of solutions to \eqref{eq:disp2} with $\eps$-oscillating sequence of initial data $(u^\eps_0)$ and assume $g^\eps(t,\cdot)$ is $\eps$-oscillating for all time $t\in\R$. Then,  $u^\eps(t,\cdot)$ also is $\eps$-oscillating for all $t\in\R$.
\end{remark}

\subsection{Localisation of Wigner measures on the critical set}

The fact that $(u^{\eps}(t,\cdot))$ is a sequence of solutions to~\eqref{eq:disp2} imposes restrictions on the measures $\mu_t$ that can be attained as a limit of their Wigner functions. In the region in the phase space $\R^d_x\times\R^d_\xi$ where equation \eqref{eq:disp2} is dispersive (\textit{i.e.} away from the critical points of $\lambda$) the energy of the sequence $(u^\eps(t,\cdot))$ is dispersed at infinite speed to infinity. These heuristics are made precise in the following result.

\begin{proposition}\label{prop:loc}
Let $(u^{\eps}(t,\cdot))$ be a sequence of solutions to \eqref{eq:disp2} issued from an $L^2(\R^d)$-bounded and $\eps$-oscillating sequence of initial data $(u^\eps_0)$,  and suppose that the measures $\mu_t$ are given by Proposition \ref{prop:msc}. Then, for almost every $t\in\R$ the measure $\mu_t$ is supported above the set of critical points of $ \lambda$ :
$${\rm supp}\, \mu_t\subset\Lambda=\{ (x,\xi) \in  \R^d\times\R^d \; : \; \nabla \lambda(\xi)=0\}.$$ 
\end{proposition}

\medskip

The result of Proposition~\ref{prop:loc} follows from a geometric argument : the fact that $u^\eps$ are solutions to \eqref{eq:disp2} translates in an invariance property of the measures $\mu_t$.
\begin{lemma}\label{lemma:inv1}
For almost every $t\in\R$, the measure $\mu_t$ is invariant by the flow 
$$\phi^1_s : \R^d\times\R^d\ni(x,\xi)\longmapsto (x+s\nabla \lambda(\xi),\xi)\in\R^d\times\R^d,\;\;s\in\R.$$
This means that for every function $a$ on $\R^d\times\R^d$ that is Borel measurable one has:
$$\int_{\R^d\times\R^d}a\circ\phi^1_s(x,\xi)\mu_t(dx,d\xi)=\int_{\R^d\times\R^d}a(x,\xi)\mu_t(dx,d\xi), \quad s\in\R.$$
\end{lemma}
This result is part of Theorem 2 in \cite{MaciaAv}. We reproduce the argument here for the reader's convenience, since we are going to use similar techniques in the sequel.

\begin{proof}[Proof of Lemma \ref{lemma:inv1}]
It is enough to show that, for all $a\in{\mathcal C}_c^\infty(\R^d\times\R^d)$ and $\theta\in{\mathcal C}_c^\infty(\R)$, the quantity 
$$R^\eps(\theta,a):=\int_{\R\times\R^d\times\R^d} \theta(t) \left.\frac{d}{ds}(a\circ\phi^1_s(x,\xi) )\right|_{s=0}W_{u^{\eps_k}(t,\cdot)}^{\eps_k}(x,\xi)dx\,d\xi\, dt$$tends to $0$ for the subsequence $\eps_k$  of Proposition~\ref{prop:msc}.
Note that
$$\left.\frac{d}{ds}(a\circ\phi^1_s)\right|_{s=0}=\nabla_\xi\lambda\cdot\nabla_x a=\{\lambda,a\};$$
therefore, by the symbolic calculus of semiclassical pseudodifferential operators, Proposition~\ref{prop:symbol}: 
\begin{eqnarray*}
\op_\eps\left(\left.\frac{d}{ds}(a\circ\phi^1_s)\right|_{s=0}\right) & = & \frac{i}{\eps} \left[ \lambda(\eps D)\;,\;\op_\eps(a)\right]+O_{{\mathcal L}\left(L^2(\R^d)\right)}(\eps)
\end{eqnarray*}
and, using the fact that $u^\eps$ solves \eqref{eq:disp2}:
$$\displaylines{
\qquad \frac{i}{\eps}\int_\R\theta(t) \left(  \left[ \lambda(\eps D),\op_\eps(a)\right]  u^\eps(t,\cdot),u^\eps(t,\cdot)\right) dt+O(\eps)
\hfill\cr\hfill
= -\eps \int_\R \theta(t) \frac{d}{dt} \left( \op_\eps(a)  u^\eps(t,\cdot),u^\eps(t,\cdot)\right) dt
= \eps  \int_\R \theta'(t)\left( \op_\eps(a)  u^\eps(t,\cdot),u^\eps(t,\cdot)\right) dt
= O(\eps).\cr}$$
This estimate together with identity \eqref{eq:wbdd} show that $R^\eps(\theta,a)=O(\eps)$, which gives the result that we wanted to prove. \end{proof}

Proposition~\ref{prop:loc} follows easily from Lemma \ref{lemma:inv1} and the following elementary fact.

\begin{lemma}\label{lem:classicdisp} Let $\Omega\subset\R^d$ and $\Phi_s:\R^d\times\Omega\To\R^d\times\Omega$ a flow satisfying: for every compact $K\subset\R^d\times\Omega$ such that $K$ contains no stationary points of $\Phi$ there exist constants $\alpha,\beta>0$ such that:
$$\alpha |s| - \beta \leqslant |\Phi_s(x,\xi)| \leqslant \alpha|s|+\beta,\quad\forall(x,\xi)\in K.$$ 
Let $\mu$ be a finite, positive Radon measure on $\R^d\times\Omega$ that is invariant by the flow $\Phi_s$. Then $\mu$ is supported on the set of stationary points of  $\Phi_s$.
\end{lemma}
\begin{proof}
It suffices to show that $\mu(K)=0$ for every compact set $K\subset\R^d\times\Omega$ as in the statement of the lemma. By the assumption made on $\Phi_s$, it is possible to find $s_k\to+\infty$ such that $\Phi_{s_k}(K)$, $k\in\N$, are mutually disjoint. The invariance property of $\mu$ implies that $\mu(\Phi_{s_k}(K))=\mu(K)$ and therefore, for every $N>0$:
$$\mu\left(\bigcup_{k=1}^N\Phi_{s_k}(K)\right)=N\mu(K).$$
Since $\mu$ is finite, we must have $\mu(K)=0$.
\end{proof} 
\medskip


 \section{Two-microlocal Wigner distributions}\label{sec:two_microlocal}

The localization result for semiclassical measures that we obtained in the preceding section is still very far from the conclusions of Theorems \ref{theo:disc} and \ref{theo:nondis}. In particular, Proposition \ref{prop:loc} does not explain how the measures $\mu_t$ depend on the sequence of initial data of the sequence of solutions $(u^\eps(t,\cdot))$. For obtaining more information, we use two-microlocal tools that  
we introduce in a rather general framework in this section.

\medskip 

From now on, we assume that $X$ is a connected, closed embedded submanifold of $(\R^d)^*$ with codimension $p>0$. Given any $\sigma\in X$, $T_\sigma X$ and $N_\sigma X$ will stand for the cotangent and normal spaces of $X$ at $\sigma$ respectively (as defined in~(\ref{def:T*X}) and~(\ref{def:NX})). The tubular neighborhood theorem (see for instance \cite{Hirsch}) ensures that there exists an open neighborhood $U$ of $\{(\sigma,0)\,:\,\sigma\in X\}\subseteq N X$ such that the map:
$$
U\ni (\sigma,v)\longmapsto \sigma+v\in(\R^d)^*,
$$
is a diffeomorphism onto its image $V$. Its inverse is given by:
$$
\begin{array}{ccl}
V\ni\xi \longmapsto (\sigma(\xi),\xi-\sigma(\xi))\in U,
\end{array}
$$
for some  smooth map $\sigma:V\longrightarrow X$. 
When $X=\{\xi_0\}$ consists of a single point, the function $\sigma$ is constant, identically equal to $\xi_0$.

\medskip

We extend the phase space $T^*\R^d:=\R^d_x\times(\R^d)^*_\xi$ with a new variable $\eta\in \overline{ \R^d}$, where $\overline{\R^d}$ is the compactification of $\R^d$ obtained by adding a sphere ${\bf S}^{d-1}$ at infinity. The test functions associated with this extended phase space are functions $a\in{\mathcal C}^\infty(T^*\R^d_{x,\xi}\times\R^d_\eta)$ which satisfy the two following properties:
\begin{enumerate}
\item There exists a compact $K \subset T^*\R^d$ such that, for all $\eta\in\R^p$, the map $(x,\xi)\mapsto a(x,\xi,\eta)$ is a smooth function compactly supported in $K$.
\item There exists a smooth function $a_\infty$ defined on $T^*\R^d\times{\bf S}^{d-1}$ and $R_0>0$ such that, if $|\eta|>R_0$, then $a(x,\xi,\eta)=a_\infty(x,\xi,\eta/|\eta|)$.  
\end{enumerate}
We denote by $\mathcal{A}$ the set of such functions and for $a\in  \mathcal{A}$ we write:
\begin{equation}\label{def:aeps}
a_\eps(x,\xi):=a\left(x,\xi,\frac{\xi-\sigma(\xi)}{\epsilon}\right).
\end{equation}
Given $f\in L^2(\R^d)$, we define the two-microlocal Wigner distribution $W^{X,\eps}_{f}\in\mathcal{D}'(\R^d\times V\times \overline{\R^d})$ by:
\begin{equation}
\left\langle W^{X,\eps}_{f},a\right\rangle :=(\op_\eps(a_\eps) f|f)_{L^2(\R^d)},\quad \forall a\in {\mathcal A}.
\end{equation}
Since $a_\eps(x,\eps\xi) = a\left(x,\eps \xi,\frac{\eps \xi-\sigma(\eps\xi)}{\epsilon}\right)$ has derivatives that are uniformly bounded in~$\eps$,  the Calderón-Vaillancourt theorem (see Appendix~\ref{sec:pdo}) gives the uniform boundedness of the family of operators $(\op_\eps(a_\eps))_{\eps>0}$ in $L^2(\R^d)$. In addition, any function $a\in {\mathcal C}^\infty_c(\R^d\times V)$ can be naturally identified to a function in ${\mathcal A}$ which does not depend on the last variable. For such $a$, one clearly has 
$$\left\langle W^{X,\eps}_{f},a\right\rangle=\int_{\R^d\times\R^d}a(x,\xi)W^\eps_f(x,\xi)dx\,d\xi.$$
Putting the above remarks together, one obtains the following.
\begin{proposition}\label{prop:2microdis} Let $(f^\eps)_{\eps>0}$ be bounded in $L^2(\R^d)$; suppose in addition that this sequence has a semiclassical measure $\mu$. Then, $(W^{X,\eps}_{f^\eps})_{\eps>0}$ is a bounded sequence in $\mathcal{D}'(\R^d\times V\times \overline{\R^d})$ whose accumulation points $\mu^X$ satisfy:
\[
\left\langle \mu^{X},a\right\rangle=\int_{\R^d\times\R^d}a(x,\xi)\mu(dx,d\xi),\quad \forall a\in {\mathcal C}^\infty_c(\R^d\times V).
\]
\end{proposition}
The distributions $ \mu^{X}$ turn out to have additional structure (they are not positive measures on $\R^d\times V\times \overline{\R^d}$, though) and can be used to give a more precise description of the restriction $\mu\rceil_{\R^d\times X}$ of semiclassical measures. The measure $\mu^{X}$ decomposes into two parts: a \textit{compact} part, which is essentially the restriction of $\mu^{X}$ to the interior $\R^d\times V\times \R^d$ of  $\R^d\times V\times \overline{\R^d}$, and \emph{a part at infinity}, which corresponds to the restriction to the sphere at infinity $\R^d\times V\times\mathbf{S}^{d-1}$. 


\subsection{The compact part}
On the neighborhood of any point $\sigma \in X$, one may find a system of $p$ equations on $\R^d$ for which $X$ is the zero set. Let $\varphi(\xi)=0$ be such a system in an open set $\Omega$ that we can assume included in the set $V$ where the map $\sigma$ is defined. Then, a parametrization of $N_\sigma X$ associated to this system of equations is 
$$N_\sigma X= \{ \, ^td\varphi(\sigma) z,\;\; z\in\R^p\}.$$
For $\sigma\in X$, we define functions of $L^2(N_\sigma X)$ as square integrable functions 
$$\R^p\ni z\mapsto f(z),$$
where $z$ is the parameter of a parametrization of $N_\sigma X$ that we fixed \emph{a priori}.

Besides, one associates with the system $\varphi(\xi)=0$ a smooth map $\xi\mapsto B(\xi)$ from the neighborhood~$\Omega$ of $\sigma$ into the set of $d\times p$ matrices such that 
\begin{equation}\label{def:B}
\xi-\sigma(\xi)=B(\xi)\varphi(\xi),\;\;\xi\in\Omega.
\end{equation}
Given a function $a\in {\mathcal C}^\infty_c(\R^d\times \Omega\times \R^d)$ and a point $(\sigma,v)\in TX$, we can use the system of coordinates $\varphi(\xi)=0$ to define an operator acting on $f\in L^2(N_\sigma X)$ given by:
$$Q_a^\varphi (\sigma,v)f(z)=\int_{\R^p\times \R^p}a\left(v+ \, ^td\varphi(\sigma) \frac{z+y}{2},\sigma,  B(\sigma) \eta\right)f(y){\rm e}^{i\eta\cdot (z-y)}\frac{d\eta\,dy}{(2\pi)^p}.$$
In other words, $Q_a^\varphi(\sigma,v)$ is obtained from $a$ by applying the non-semiclassical Weyl quantization to the symbol 
$$(z,\eta)\mapsto a\left(v+ \, ^td\varphi(\sigma) z ,\sigma,  B(\sigma)  \eta\right)\in {\mathcal C}^\infty_c(\R^p\times\R^p).$$
We write
$$Q_a^\varphi (\sigma,v)= a^W\left(v+ \, ^td\varphi(\sigma) z,\sigma,  B(\sigma) D_z\right).$$

\medskip 

If one changes the system of coordinates into $\widetilde\varphi(\xi)=0$ on some open neighborhood~$\widetilde\Omega$ of~$\sigma$, then, there exists  a smooth  map $R(\xi)$ defined on the open set $\Omega\cap\widetilde\Omega$ (where both system of coordinates can be used), and valued in the set of invertible $p\times p$ matrices, such that $\widetilde \varphi(\xi)=R(\xi) \varphi(\xi)$.
One then observe that the matrix $\widetilde B(\xi)$ associated with the choice of $\widetilde \varphi$ is given by $\widetilde B(\xi)=B(\xi)R(\xi)^{-1}.$ Besides, for $a\in {\mathcal C}^\infty_c(\R^d\times ( \Omega\cap \widetilde\Omega) \times \R^d)$,
\begin{eqnarray*}
 Q_a^{\widetilde \varphi} (\sigma,v)& = &
\int_{\R^p\times \R^p}a\left(v+ \, ^td\widetilde\varphi(\sigma) \frac{z+y}{2},\sigma,  \widetilde B(\sigma) \eta\right)f(y){\rm e}^{i\eta\cdot (z-y)}\frac{d\eta\,dy}{(2\pi)^p}\\
& = & \int_{\R^p\times \R^p}a\left(v+ \, ^td\varphi(\sigma) \,^tR(\sigma) \frac{z+y}{2},\sigma,  B(\sigma)R(\sigma)^{-1} \eta\right)f(y){\rm e}^{i\eta\cdot (z-y)}\frac{d\eta\,dy}{(2\pi)^p}.
\end{eqnarray*}
We obtain 
$$ Q_a^{\widetilde \varphi} (\sigma,v)= U(\sigma)Q_a^{\varphi}(\sigma,v)U^*(\sigma),$$
where $U(\sigma)$ is the unitary operator of $L^2(N_\sigma X)\sim L^2(\R^p)$ associated with the linear map from~$\R^p$ into itself :
$ z\mapsto \, ^tR(\sigma) z$. More precisely, 
$$\forall f\in L^2(\R^p),\;\; U(\sigma) f(z)= \left|{\rm det} \,R(\sigma)\right|^{p\over 2} f(\,^tR(\sigma)z).$$
This map is the one associated with the change of parametrization on $N_\sigma X$ induced by turning~$\varphi$ into $\widetilde\varphi$, and the 
 map $(z,\zeta)\mapsto (\,^tR(\sigma)z,R(\sigma)^{-1}\zeta)$ is a symplectic transform of the cotangent of $\R^p$. This is the standard rule of transformation of pseudodifferential operators through linear change of variables (see \cite{AlinhacGerard} for an example or any textbook about pseudodifferential calculus).  

\medskip 

Because of this invariance property with respect to the change of system of coordinates, we shall say that $a$ defines an operator $Q_a(\sigma,v)$ on $L^2(N_\sigma X)$.
Clearly, $Q_a(\sigma,v)$ is smooth and compactly supported in $(\sigma,v)$; moreover, $Q_a(\sigma,v)\in\mathcal{K}(L^2(N_\sigma X))$, for every $(\sigma,v)\in TX$, 
where $\mathcal{K}(L^2(N_\sigma X))$ stands for the space of compact operators on $L^2(N_\sigma X)$.

\begin{proposition}\label{prop:compact}
Let $\mu^X$ be given by Proposition \ref{prop:2microdis}. Then there exist a positive measure $\nu$ on $T^*X$ and a measurable family:
\[
M:T^*X\ni(\sigma,v)\longmapsto M(\sigma,v)\in \mathcal{L}^1_+(L^2(N_\sigma X)),
\]
satisfying
\[
\Tr_{L^2(N_\sigma X)} M(\sigma,v)=1,\quad \text{ for }\nu\text{-a.e. }(\sigma,v)\in T^*X,
\]
and such that, for every $a\in {\mathcal C}^\infty_c(\R^d\times V\times \R^d)$ one has:
$$\left\langle \mu^{X},a\right\rangle = \int_{T^*X}\Tr_{L^2(N_\sigma X)}(Q_a(\sigma,v)M(\sigma,v))\nu(d\sigma,dv).$$
\end{proposition}

\begin{proof}
We suppose that we are given a local system of $p$ equations of $X$ by $\varphi(\xi)=0$. Put $\xi = (\xi',\xi'') \in \R^p \times \R^{d-p}$. Without loss of generality, we may assume that $d_{\xi'} \varphi(\xi)$ is invertible. We consider the smooth valued function $B$ satisfying  $\xi-\sigma(\xi)=B(\xi) \varphi(\xi)$ and we introduce the local diffeomorphism
$$\Phi : \left(\varphi(\xi),\xi''\right)\mapsto \xi .$$
Note that if $\xi=\Phi(\zeta)$, $\zeta=(\zeta',\zeta'')$, we have $\zeta'=\varphi(\xi)=\varphi(\Phi(\zeta))$ and $\zeta''=\xi''$.
We use this diffeomorphism according to the next lemma.

\begin{lemma}\label{lem:geometry}
For all $f\in L^2(\R^d)$ and $a\in{\mathcal A}$,
$$  \left(\op_\eps(a_\eps)f\;,\;f\right) = \left( \op_\eps\left(a\left(^t d\Phi(\xi)^{-1} x,\Phi(\xi),B\left(\Phi(\xi)\right) {\xi'\over\eps}\right)\right) {\mathcal U}_\eps f\;,\;{\mathcal U}_\eps f\right) + O(\eps) \| f\| ^2$$
where $f\mapsto {\mathcal U}_\eps f$ is an isometry of $L^2(\R^d)$.
\end{lemma}

The proof of this lemma is in the Appendix~\ref{sec:app_proof}.
This lemma reduces the problem  to the analysis of the concentration of the bounded family $\widetilde f^\eps =({\mathcal U}_\eps f)$ on the submanifold $\Lambda_0=\{\xi'=0\}$ which has the additional property to be a vector space. This special case  has been studied in~\cite{CFMProc} where it is proved (see pages 96-97, Proposition~2)  that up to a subsequence, there exist a positive measure $\nu_0$ on $T^*\R^{d-p}$ and a measurable family of trace~$1$ operators:
\[
M_0:T^*\R^{d-p}\ni(\sigma,v)\longmapsto M_0(\sigma,v)\in \mathcal{L}^1_+(L^2(\R^p)),
\]
satisfying for any $b\in{\mathcal C}_c^\infty(\R^{2d+p})$,
$$\displaylines{\qquad \lim_{\eps\rightarrow 0} \left(\op_\eps(b_\eps)\widetilde f^\eps\;,\;\widetilde f^\eps\right) 
\hfill\cr\hfill
 =\int_{\R^{d-p}\times\R^{d-p}} {\rm Tr}_{L^2(\R^p)} \Bigl( b^W\left((z,u''),(0,\theta''), D_z\right) 
\, M_0(u'',\theta'')\Bigr)d\nu_0(du'',d\theta'').\cr}$$
The reader will find in Appendix~\ref{sec:ovm} comments on the operator-valued families. 
Therefore, for compactly supported $a\in{\mathcal A}$, and choosing  $b(x,\xi,\eta')= a\left(^t d\Phi(\xi)^{-1} x,\Phi(\xi),B\left(\Phi(\xi)\right)\eta'\right)$, one obtains 
$$\displaylines{ \lim_{\eps\rightarrow 0} \left(\op_\eps(a_\eps)f^\eps\;,\;f^\eps\right)  \hfill\cr\hfill
=\int_{\R^{d-p}\times\R^{d-p}} {\rm Tr}_{L^2(\R^p)} \Bigl( a^W\left(\,^t d\Phi(0,\theta'')^{-1}(z,u''), \Phi(0,\theta''),B(\Phi(0,\theta'')) D_z\right) 
\hfill\cr\hfill \times\, M_0(u'',\theta'')\Bigr)d\nu_0(du'',d\theta'').\qquad\cr}$$
Note that the map $\theta''\mapsto \sigma =\Phi(0,\theta'')$ is a parametrization of $X$ with associated parametrization of $T^*X$,
$$(\theta'',u'')\mapsto (\sigma,v)=\left( \Phi(0,\theta''), ^t d\Phi(0,\theta'')^{-1}(0,u'')\right).$$
Since the Jacobian of this mapping is $1$, after a change of variable, we obtain an operator valued measurable family $M$ on $T^*X$ and a measure $\nu$ on $T^*X$
such that 
$$\displaylines{  \lim_{\eps\rightarrow 0}\left(\op_\eps(a_\eps)f\;,\;f\right) \hfill\cr\hfill = \int_{T^*X}{\rm Tr}_{L^2(\R^p)} \left( a^W\left(\,^t d\Phi(0,\theta''(\sigma))^{-1}(z,0) +v , \sigma,B(\sigma) D_ z\right) M(\sigma,v)\right)d\nu(d\sigma,dv).\cr}$$
We now take advantage of the fact that 
$\varphi(\Phi(\zeta))=\zeta'$ for all $\zeta\in\R^d$
in order to write 
$$d\varphi(\Phi(\zeta))d \Phi(\zeta)=({\rm Id}, 0).$$
We deduce 
$$\forall z\in \R^p,\;\;  \,^td\Phi(\zeta) \,^td\varphi(\Phi(\zeta)) z=(z,0),$$
which implies 
$$\forall z\in \R^p,\;\;  \,^t d\varphi(\Phi(\zeta)) z= \,^td \Phi(\zeta)^{-1} (z,0).$$
Therefore,
\begin{eqnarray*}
 \lim_{\eps\rightarrow 0} \left(\op_\eps(a_\eps)f\;,\;f\right)& =& \int_{T^*X} {\rm Tr}_{L^2(\R^p)} \left( a^W\left(\,^t d\varphi(\sigma)z +v , \sigma,B(\sigma) D_z\right) M(\sigma,v)\right)d\nu(d\sigma,dv)\\
  & = & \int_{T^*X} {\rm Tr}_{L^2(N_\sigma X)} \left( Q_a(\sigma)M(\sigma,v)\right)d\nu(d\sigma,dv).
  \end{eqnarray*}
\end{proof}


\subsection{Measure structure of the part at infinity}

To analyze the part at infinity, we use 
 a cut-off function $\chi\in{\mathcal C} _c^\infty (\R^d)$ such that $0\leq \chi\leq 1$, $\chi(\eta)=1$ for $|\eta|\leq 1$ and  $\chi(\eta)=0$ for $|\eta|\geq 2$, and we write 
$$\left\langle W^{X,\eps}_{f},a\right\rangle=\left\langle W^{X,\eps}_{f},a_{R}\right\rangle+\left\langle W^{X,\eps}_{f},a^{ R}\right\rangle,$$
with
\begin{equation}\label{def:aR}
a_{R}(x,\xi,\eta):=a(x,\xi,\eta) \chi\left(\frac{\eta}{R}\right)\;{\rm and}\;\;a^{R}(x,\xi,\eta):=a(x,\xi,\eta) \left(1-\chi\left(\frac{\eta}{R}\right)\right).
\end{equation}
Observe that $a_R$ is compactly supported in all variables. We thus focus on the second part, and more precisely on the quantity 
$$\limsup_{R\to\infty}\, \limsup_{\eps\to 0^+}\left\langle W^{X,\eps}_{f},a^R\right\rangle .$$

\medskip 

 We denote by $S\Lambda$ the compactified normal bundle to $\Lambda$, viewed as a submanifold of $\R^d\times \R^d$, the fiber of which  is $T^*_\sigma\R^d\times S_\sigma\Lambda$ above $\sigma$ with $S_\sigma\Lambda$ being obtained by taking the quotient of $N_\sigma\Lambda$ by the action of  $\R^*_+$ by homotheties.

\begin{proposition}
Let $(f^\eps)$ be a bounded family of $L^2(\R^d)$.
There exists a subsequence~$\eps_k$  and a measure $\gamma$ on $S\Lambda$ such that for all $a\in{\mathcal A}$,
$$\displaylines{\lim_{R\to\infty} \lim_{k\to +\infty}\left\langle W^{X,\eps_k}_{f^{\eps_k}},a^{R}\right\rangle
=\int_{\R^d\times X\times {\bf S}^{d-1}} a_\infty (x,\sigma,\omega) \gamma(dx,d\sigma,d\omega) 
\hfill\cr\hfill
+ 
\int_{\R^d\times X^c\times {\bf S}^{d-1}} a_\infty \left(x,\xi,{\xi-\sigma(\xi)\over|\xi-\sigma(\xi)|}\right)\mu(dx,d\xi),\cr}$$
where $X^c$ denotes the complement of the set $X$ in $\R^d$. 
\end{proposition}

\begin{proof} We begin by recalling the arguments that prove the existence of the measure~$\gamma$, which are the same that the one developed in the vector case in~\cite{CFMProc}.
Since $a=a_\infty$ for~$|\eta|$ large enough, we have $a^R=a_\infty^R$ as soon as $R$ is large enough and the quantity 
$$ \limsup_{R\to\infty}\, \limsup_{\eps\to 0^+}\left\langle W^{X,\eps}_{f^\eps},a^{R}\right\rangle$$
will only depend on $a_\infty$.  
 Therefore, by considering a dense subset of ${\mathcal C}_c(T^*\R^{d}\times {\bf S}^{d-1})$, we can find a subsequence $(\eps_k)$ by a diagonal extraction process such that the following linear form on ${\mathcal C}_c(T^*\R^{d}\times {\bf S}^{d-1})$ is well-defined
 $$\ell : a_\infty\mapsto \lim_{R\to\infty} \lim_{k\to +\infty}\left\langle W^{X,\eps_k}_{f^{\eps_k}},a^{R}\right\rangle.$$
 We then observe that 
 $$\forall \alpha,\beta\in\N^d,\;\;\exists C_{\alpha,\beta}>0,\;\;\sup_{\R^{2d}}\left| \partial_x^\alpha\partial_\xi^\beta \left(a^{R}\right)_\eps\right|\leq C_{\alpha,\beta}(\eps^{|\beta|}+R^{-|\beta|}).$$
 This implies that the symbolic calculus on symbols $(a^R)_\eps$ is semiclassical with respect to the small parameter $\sqrt{\eps^2+R^{-2}}$. To be precise, one has the following  weak G\aa rding inequality: if $a\geq 0$, then, for all $\kappa>0$,  there exists a constant~$C_\kappa$ such that 
 $$\left\langle W^{X,\eps}_{f^\eps},a^{R}\right\rangle \geq -\left( \kappa + C_\kappa \left(\eps+{1\over R}\right) \right)\| f^\eps\|^2_{L^2(\R^d)}.$$
 We then conclude that the linear form $\ell$ defined above is positive and defines a positive Radon measure~$\widetilde \rho$. It remains to compute~$\widetilde\rho$ outside $X$. In this purpose, we set 
 $$a^R=a^R_\delta+a^{R,\delta}\;\;{\rm with}\;\; a^{R}_{\delta}(x,\xi,\eta)=  a^R(x,\xi,\eta)(1-\chi)\left({\xi-\sigma(\xi)\over \delta}\right)$$ and we observe that, by the definition of $\mu$:
 $$\lim_{\delta\rightarrow 0} \,  \lim_{R\to\infty}\,\lim_{\eps\rightarrow 0} \left\langle W^{X,\eps_k}_{f^{\eps_k}},a^{R}_{\delta}\right\rangle
 = \int_{\R^d\times X^c\times {\bf S}^{d-1}} a_\infty \left(x,\xi,{\xi-\sigma(\xi)\over|\xi-\sigma(\xi)|}\right)\mu(dx,d\xi),$$
 which concludes the proof of the existence of the measure $\gamma$.
 
 \medskip 
 
 Let us now analyze the geometric properties of this measure. We choose a system of local coordinates of $\Lambda$ and introduce the matrix $B$ as in (\ref{def:B}). By Lemma~\ref{lem:geometry} and the result of~\cite{CFMProc} for vector spaces: up to a subsequence, there exists a measure $\widetilde\gamma_0$ on $\R^d\times \R^{d-p} \times {\bf S}^{p-1}$ such that 
 $$\displaylines{\lim_{\delta\to 0^+}\lim_{R\to\infty} \lim_{\eps\to 0^+}\left\langle W^{X,\eps}_{f},a^{R,\delta}\right\rangle\hfill\cr\hfill=
 \int_{\R^d\times \R^{d-p} \times {\bf S}^{p-1}} a_\infty\left(^t d\Phi(0,\xi'')^{-1} x,\Phi(0,\xi''),{B\left(\Phi(0,\xi'')\right) \omega \over|B\left(\Phi(0,\xi'')\right) \omega |}\right)\widetilde\gamma_0(dx,d\xi,d\omega).\cr}$$
 The mapping $\xi''\mapsto \Phi(0,\xi'')$ is a parametrization of $X$ and the mapping 
 $$(x,\xi)\mapsto \left(^t d\Phi(0,\xi'')^{-1} x,\Phi(0,\xi'')\right)$$
  is the associated mapping of $T^*_X\R^d$.
 Therefore, this relation   defines a measure $\widetilde\gamma$ on $T^*X\times {\bf S}^{p-1}$ such that  
 \begin{equation}\label{eq:tildegamma}
 \lim_{\delta\to 0^+}\lim_{R\to\infty} \lim_{\eps\to 0^+}\left\langle W^{X,\eps}_{f},a^{R,\delta}\right\rangle
 =
 \int_{T^*X\times  {\bf S}^{p-1}} a_\infty\left(x,\sigma,{B\left(\sigma\right) \omega \over|B\left(\sigma\right) \omega |}\right)\widetilde\gamma(dx,d\xi,d\omega). 
 \end{equation}
 Besides, using that 
 \begin{equation}\label{dsigma}
 {\rm Id} = d\sigma(\sigma_0) + B(\sigma_0) d\varphi(\sigma_0)
 \end{equation} for any $\sigma_0\in X$, we deduce that for any~$\zeta\in T_{\sigma_0}\R^d$, we have the decomposition 
$$\zeta= d\sigma(\sigma_0) \zeta + B(\sigma_0) d\varphi(\sigma_0)\zeta,\;\;{\rm with}\;\;d\sigma(\sigma_0) \zeta \in T_\sigma X\;\;{\rm and }\;\;B(\sigma_0) d\varphi(\sigma_0)\zeta\in N_{\sigma_0} X.$$
Now, since $d\varphi$ is of rank $p$, one can write any $\omega\in {\bf S}^{p-1}$ as $\omega=d\varphi(\sigma_0)\zeta$ and  the points $B(\sigma_0)\omega$ are in $N_{\sigma_0} X$.
By identification of $\gamma$ in~(\ref{eq:tildegamma}), 
we deduce that $\gamma(x,\sigma,\cdot)$ is a measure on the set 
$$\left\{ {B\left(\sigma\right) \omega \over|B\left(\sigma\right) \omega |},\;\;\omega\in {\bf S}^{p-1}\right\}= \faktor{N_\sigma X}{\R^*_+}=S_\sigma X,$$
which completes the proof of the proposition. 
\end{proof}


\section{Two microlocal Wigner measures and families of solutions to dispersive equations}\label{sec:4}

We now consider families of solutions to equation~(\ref{eq:disp2}). 
As proved in Proposition~\ref{prop:loc}, the Wigner measure of the family $(u^\eps(t,\cdot))$ concentrates on the set
$\Lambda=\{\nabla\lambda(\xi)=0\}.$
In order to analyze~$\mu^t$ above $\Lambda$, we perform a second microlocalization above the set $X=\Lambda$, with average in time. We  consider 
for $\theta\in L^1(\R)$ the quantities
$$\int_\R \theta(t) \left\langle W^{\Lambda,\eps}_{u^\eps(t,\cdot)},a\right\rangle dt$$
for symbols $a\in{\mathcal A}$. 
Up to extracting a subsequence $\eps_k$, we construct $L^\infty$ maps 
$$t\mapsto \gamma_t(dx,d\sigma,d\omega),\;\; t\mapsto \nu_t(d\sigma,dv),\;\; t\mapsto M_t(\sigma,v)$$
valued respectively on the set of positive Radon measures on $\R^d\times\Lambda\times {\bf S}^{d-1}$, on the set of  positive Radon measures on $T^*\Lambda$ and finally on the set of measurable families from~$T^*\Lambda$ onto the set of positive trace class operators on $L^2(N\Lambda)$, such that for all $\theta\in L^1(\R)$ and for all $a\in{\mathcal A}$:
$$\displaylines{\int_\R \theta(t) \left\langle W^{\Lambda,\eps_k}_{u^{\eps_k}(t,\cdot)},a\right\rangle dt\Tend{k}{+\infty}
\int_\R  \int _{\R^d\times \Lambda\times {\bf S}^{d-1}} \theta(t)a_\infty (x,\sigma,\omega) \gamma_t(dx,d\sigma,d\omega) dt\hfill\cr\hfill+ 
\int_\R \int_{T^*\Lambda}\theta(t) {\rm Tr}_{L^2(N_\sigma\Lambda)} (Q_a(\sigma,v)M_t(\sigma,v)\nu_t(d\sigma,dv)dt.\cr}
$$
The measures $\gamma^t$ and $\nu^t$, and the map $M^t$ satisfy additional properties coming from the fact that the family $(u^\eps(t,\cdot))$ solves a time-dependent equation. 
These properties are discussed in the next two sections. We shall see that  the measures $\gamma_t$ are invariant under a linear flow and that we can choose the sequence $\eps_k$ such that the map $t\mapsto M_t$ is continuous (and even ${\mathcal C}^1$).

\subsection{Transport properties of the compact part}

Since $\Lambda$ is the set of critical points of~$\lambda$, the matrix $d^2 \lambda$ is intrinsically defined above points of $\Lambda$. Thus, using the formalism of the preceding sections, 
$$Q_{d^2\lambda(\sigma)\eta\cdot\eta}= d^2\lambda(\sigma)D_z\cdot D_z.$$

\begin{proposition}
The map $t\mapsto \nu_t$ is constant and the map $$t\mapsto M_t(\sigma,v)\in \mathcal{C}(\R;\mathcal{L}_+^1(L^2(N_\sigma\Lambda))$$ solves the Heisenberg equation~(\ref{eq:heis1}).
\end{proposition}

\begin{proof}
We analyze for $a\in\mathcal{C}_c^\infty(\R^{3d})$ the time evolution of the quantity
$ \left\langle W^{\Lambda,\eps}_{u^{\eps}(t,\cdot)},a\right\rangle.$ We have 
$${d\over dt} \left\langle W^{\Lambda,\eps}_{u^{\eps}(t,\cdot)},a\right\rangle ={1\over i\eps^2} \left(\left[ \op_\eps(a_\eps),\lambda(\eps D)\right] u^\eps(t,\cdot),u^\eps(t,\cdot)\right) + {1\over i} 
 \left(\left[ \op_\eps(a_\eps),V_{\rm ext}\right] u^\eps(t,\cdot)\;,\;u^\eps(t,\cdot)\right)+O(\eps).$$
By standard symbolic calculus for Weyl quantization, we have  in ${\mathcal L}(L^2(\R^d))$
$${1\over i\eps^2}\left[ \op_\eps(a_\eps),\lambda(\eps D)\right] = {1\over \eps}  \,\op_\eps(\nabla\lambda(\xi) \cdot\nabla_x a_\eps) +O(\eps).$$
Besides, by Taylor formula and by use of $\nabla\lambda (\sigma(\xi))=0$, we have 
\begin{equation}\label{taylor_lambda}
\nabla \lambda(\xi) =  d^2\lambda(\sigma(\xi)) \left(\xi-\sigma(\xi)\right)   + \Gamma(\xi) \left(\xi-\sigma(\xi) \right)\cdot\left(\xi-\sigma(\xi)\right),
\end{equation}
where $\Gamma$ is a smooth matrix. This yields 
$${1\over \eps} \nabla\lambda(\xi) \cdot\nabla_x a_\eps(x,\xi) = b_\eps(x,\xi)$$
with 
$$
b(x,\xi,\eta)=d^2\lambda(\sigma(\xi))\eta\cdot \nabla_xa (x,\xi,\eta) + \Gamma(\xi) \left(\xi-\sigma(\xi) \right)\cdot\eta \nabla_x a(x,\xi,\eta).$$
At this stage of the proof, we see that ${d\over dt} \left\langle W^{\Lambda,\eps}_{u^{\eps}(t,\cdot)},a\right\rangle$ is uniformly bounded in~$\eps$, thus using a suitable version of Ascoli's theorem and a standard diagonal extraction argument, we can find a sequence $(\eps_k)$ such that the limit exists for all $a\in{\mathcal C}_c^\infty(\R^{3d})$ and all time $t\in[0,T]$ (for some $T>0$ fixed) with a limit that is a continuous map in time. The transport equation that we are now going to prove shall guarantee the independence of the limit from~$T>0$.

\medskip

We observe that for any local system  of equations of $\Lambda$, $\varphi(\xi)=0$, the operator $Q_b^\varphi$ satisfies for $(\sigma,v)\in T\Lambda$,
\begin{eqnarray*}
Q_b^\varphi (\sigma,v) & = &  b^W\left(v+\,^td\varphi(\sigma) z, \sigma, B(\sigma)D_z\right) \\
& = &{\rm op}_1\left( d^2\lambda(\sigma)B(\sigma) \eta \cdot \nabla_xa (v+\,^td\varphi(\sigma) z,\sigma,B(\sigma)\eta) \right).
\end{eqnarray*}
On the other hand, we observe   that, setting 
$$\theta(\xi,\eta) = {1\over 2} d^2\lambda(\xi) \eta\cdot\eta,$$
we have 
\begin{eqnarray}\nonumber 
 i \left[ Q^\varphi_\theta(\sigma),Q_{a}^\varphi (\sigma,v)\right] &
= & i\left[ \,^t B(\sigma)d^2\lambda(\sigma)B(\sigma) D_z\cdot D_z \;,\; Q_a^\varphi(\sigma,v)\right]\\
\label{eq:com_hessienne}
&=& {\rm op}_1 \left(\,^t d\varphi(\sigma) \, ^tB(\sigma)d^2\lambda(\sigma)B(\sigma)\eta\cdot 
\nabla_x a (v+\,^td\varphi(\sigma) z,\sigma,B(\sigma)\eta) \right),
\end{eqnarray}
and we now focus on the matrix $\,^t d\varphi(\sigma) \, ^tB(\sigma)d^2\lambda(\sigma)B(\sigma)$, and thus on the properties of the hessian $d^2\lambda(\sigma)$.

\medskip

For $\xi \in \Lambda$, the bilinear form $d^2\lambda(\xi)$ is defined intrinsically on~$T_\xi\R^d$ and  
$d^2\lambda(\xi) =0$ on $T_\xi \Lambda$. 
We deduce from~(\ref{dsigma}) that any $\zeta\in T_\xi\R^d$ satisfies 
$$\zeta= d\sigma(\xi) \zeta + B(\xi) d\varphi(\xi)\zeta\;\;{\rm with}\;\;d\sigma(\xi) \zeta \in T_\sigma\Lambda.$$
Therefore, 
$$\forall \xi\in\Lambda,\;\;d^2\lambda(\xi) = d^2\lambda(\xi) B(\xi) d\varphi(\xi).$$
Taking into account this information, equation~(\ref{eq:com_hessienne}) becomes 
$$\displaylines{
 i \left[ Q^\varphi_\theta(\sigma),Q_{{a}}^\varphi (\sigma,v)\right] 
= {\rm op}_1 \left(d^2\lambda(\sigma)B(\sigma)\eta\cdot 
\nabla_x a (v+\,^td\varphi(\sigma) z,\sigma,B(\sigma)\eta) \right).\cr}$$We conclude 
$$Q_b^\varphi (\sigma,v) = i \left[ Q^\varphi_\theta(\sigma),Q_{{a}}^\varphi (\sigma,v)\right].$$
This implies that 
$$i\partial_t (M_t(\sigma,v)\nu_t(d\sigma,dv)) =\left[\dfrac{1}{2} d^2\lambda(\sigma)D_z\cdot D_z + m_{V_{\rm ext}(t,\cdot)}(v,\sigma), M_t(\sigma,v)\right] \nu_t(d\sigma,dv).$$
Taking the trace, we get $\partial_t \nu_t=0$, thus $\nu_t$ is equal to some constant measure $\nu$ and $M_t$ satisfies equation~(\ref{eq:heis1}),
which proves the proposition. \end{proof}


\subsection{Invariance and localization of the measure at infinity}

We are concerned with the property of the $L^\infty$-map $t\mapsto \gamma^t(dx,d\sigma,d\omega)$ valued in the set of positive Radon measures on $S\Lambda$. 
We now define a flow on $S \Lambda$ by setting for $s\in\R$  
$$\phi^s_2 : (x,\sigma,\omega)\mapsto (x+s \,d^2\lambda(\sigma)  \omega,\sigma,\omega).$$

\begin{proposition}\label{prop:inv2micro}
The measure $\gamma^t$  is invariant by the flow $\phi^s_2$.
\end{proposition}

\begin{proof}
The proof essentially follows the lines of the proof of Theorem~2.5  in~\cite{AFM:15}. 
We use the cut-off function $\chi$ introduced before and set
$$a^{R,\delta}(x,\xi,\eta)=  a(x,\xi,\eta)\,\chi\left({\xi-\sigma(\xi)\over \delta}\right)\left(1-\chi\left({\eta\over R}\right)\right);$$
we introduce the smooth symbol
$$b^{R,\delta}_s(x,\xi,\eta)= a^{R,\delta}\left(x+s d^2\lambda(\xi) {\eta\over |\eta|},\xi,\eta\right),$$
which satisfies $(b_s^{R})_\infty = a_\infty\circ\phi^s_2.$
Using equation~(\ref{taylor_lambda}), we obtain  
$$\left( b^{R,\delta} _s\right)_\eps (x,\xi) =a^{R,\delta}\left(x+{s\over |\xi-\sigma(\xi)|} \nabla\lambda(\xi) ,\xi,{\xi-\sigma(\xi)\over\eps}\right) +\delta \, r_\eps^{R,\delta}(x,\xi)$$
where for all multi-index $\alpha,\beta \in \N^d$, there exists a constant $C_{\alpha,\beta}>0$ such that $r_\eps^{R,\delta}$ satisfies:
$$\sup_{x,\xi\in\R^d}\left|\partial_x^\alpha\partial_\xi^\beta   r^{R,\delta}_\eps\right| \leq C_{\alpha,\beta}.$$ 
As a consequence, 
$\langle W_{u^\eps(t,\cdot)}^{\Lambda,\eps}, r_\eps ^{R,\delta}\rangle $ is uniformly bounded in $R,\delta,\eps$ and: 
$$ \langle W_{u^\eps(t,\cdot)}^{\Lambda,\eps}, b^{R,\delta}_s \rangle = \langle W_{u^\eps(t,\cdot)}^{\Lambda,\eps},  \widetilde b^{R,\delta}_s\rangle +O(\delta),$$
uniformly with respect to $R$ and $\eps$, with 
$$\widetilde b^{R,\delta}_s(x,\xi,\eta)=a^{R,\delta}\left(x+{s\over |\xi-\sigma(\xi)|} \nabla\lambda(\xi) ,\xi,\eta\right) . $$
Note that this symbol is smooth because $|\xi-\sigma(\xi)|>R\,\eps $ on the support of $a^{R,\delta}$. 
We are going to prove that for all $\theta\in{\mathcal C}_c^\infty(\R)$, 
$$\lim_{\delta\to 0^+}\lim_{R\to\infty} \lim_{\eps\to 0^+}\int_\R\theta(t) {d\over ds} \langle W_{u^\eps(t,\cdot)}^{\Lambda,\eps}, \widetilde b^{R,\delta}_s\rangle dt =0.$$
Indeed, by the calculus of the preceding section, we have 
$${d\over ds} \langle W_{u^\eps(t,\cdot)}^{\Lambda,\eps}, \widetilde b^{R,\delta}_s\rangle= \langle W_{u^\eps(t,\cdot)}^{\Lambda,\eps}, \nabla\lambda\cdot \nabla_x c^{R,\delta}_s\rangle$$
with 
$$c^{R,\delta}_s (x,\xi,\eta) ={1\over |\xi-\sigma(\xi)| }a^{R,\delta}\left(x+{s\over |\xi-\sigma(\xi)|} \nabla\lambda(\xi) ,\xi,\eta\right).$$
The symbol $c^{R,\delta}_s$ is such that for all multi-index $\alpha \in \N^d$, there exists a constant $C_{\alpha}>0$ for which: 
$$\sup_{x,\xi\in\R^d}\left|\partial_x^\alpha   (c_s^{R,\delta})_\eps\right| \leq C_{\alpha} (R\eps)^{-1}.$$ 
This implies in particular: 
$$\left\|  \op_\eps((c^{R,\delta}_s)_\eps)\right\|_{{\mathcal L}(L^2(\R^d))}\leq {C\over R\eps}.$$
By symbolic calculus, we have
$${1\over i\eps} \left[ \op_\eps((c^{R,\delta}_s)_\eps),\lambda(\eps D)\right] = \op_\eps\left(\nabla\lambda(\xi) \cdot \nabla_x (c^{R,\delta}_s)_\eps\right)+O\left({\eps\over R}\right).$$
We deduce that for all $\theta\in{\mathcal C}_c^\infty(\R)$, 
\begin{eqnarray*}
\int_\R \theta(t) &{d\over ds}&  \langle W_{u^\eps(t,\cdot)}^{\Lambda,\eps}, \widetilde b^{R,\delta}_s\rangle dt \\
&=& \int _\R\theta(t) \left( {1\over i\eps} \left[ \op_\eps((c^{R,\delta}_s)_\eps),\lambda(\eps D)\right] u^\eps(t,\cdot)\;,\;u^\eps(t,\cdot) \right) dt +O\left({\eps\over R}\right)\\
&=& \int _\R\theta(t) \left( {1\over i\eps} \left[ \op_\eps((c^{R,\delta}_s)_\eps),\lambda(\eps D)+\eps^2 V_{\rm ext}(t,x)\right] u^\eps(t,\cdot)\;,\;u^\eps(t,\cdot) \right) dt +O\left({1\over R}\right)\\
&=&-\eps \int _\R\theta(t) {d\over dt} \left(\op_\eps((c^{R,\delta}_s)_\eps) u^\eps(t,\cdot)\;,\;u^\eps(t,\cdot) \right) dt +O\left({1\over R}\right)\\
&=& O(\eps)+O\left({1\over R}\right).
\end{eqnarray*}
As a conclusion,
\begin{eqnarray*}
 \langle W_{u^\eps(t,\cdot)}^{\Lambda,\eps}, b^{R,\delta}_s \rangle& =& \langle W_{u^\eps(t,\cdot)}^{\Lambda,\eps},  \widetilde b^{R,\delta}_s\rangle +O(\delta)\\
 & = &  \langle W_{u^\eps(t,\cdot)}^{\Lambda,\eps},  \widetilde b^{R,\delta}_0\rangle + O(|s|\eps)+O(|s| R^{-1}) + O(\delta)\\
 & = &  \langle W_{u^\eps(t,\cdot)}^{\Lambda,\eps},   b^{R,\delta}_0\rangle + O(|s|\eps)+O(|s| R^{-1}) + O(\delta),
 \end{eqnarray*}
 which implies  the Proposition.\end{proof}

 \subsection{Proofs of Theorems~\ref{theo:disc} and~\ref{theo:nondis}}
 Remind that Theorem~\ref{theo:nondis} implies Theorem~\ref{theo:disc}, thus 
 we focus on~Theorem~\ref{theo:nondis}. We first observe that the measure $\gamma_t$ is zero. Indeed, by {\bf H2};  
for $\sigma\in\Lambda$,   $d^2\lambda(\sigma)$ is one to one on $N_\sigma\Lambda$. Therefore, since~$\gamma_t$ is a measure on~$S\Lambda$, the invariance property of Proposition~\ref{prop:inv2micro} and an argument similar to the one of Lemma~\ref{lem:classicdisp} yields that~$\gamma_t=0$. As a consequence, the semi-classical measure~$\mu_t$ is only given by the compact part and one has for any $a\in{\mathcal C}_c^\infty(\R^{2d})$ and $\theta\in L^1(\R)$,
$$\int_\R \theta(t)\int_{\R^{2d} }a(x,\xi) \mu^t(dx,d\xi) =
\int_\R \theta(t) \int_{T^* \Lambda} {\rm Tr}_{L^2(N_\sigma\Lambda)}\left(Q_a(\sigma,v)M_t(d\sigma,dv)\right)dt. $$
Then, taking $\theta=\mathds{1}_{[a,b]}$ for $a,b\in\R$, $a<b$, and in view  of Proposition~\ref{prop:msc} and of Lemma~\ref{remark:tx}, we deduce that for every  every $\phi\in \mathcal{C}_c(\R^d)$ one has for the subsequence defining $M_t$ and $\nu_t$:
$$
\lim_{\eps\rightarrow 0}\int_a^b\int_{\R^d}\phi(x)|u^{\eps}(t,x)|^2dxdt=\int_a^b \int_{T^*\Lambda}\Tr_{L^2(N_\xi\Lambda)}\left[Q_\phi(v,\xi)M_t(v,\xi)\right]\nu(dv,d\xi)dt,$$
where $M_t$ satisfies~(\ref{eq:heis1}). This concludes the proof of Theorem~\ref{theo:nondis}. We emphasize that the measure $\nu$ and the operator valued family $M_0$ are utterly determined by the initial data.


\section{Bloch projectors and semiclassical measures}\label{sec:5}

In this section we prove Theorem~\ref{mainresult}, as a result of the analysis in Section~\ref{sec:two_microlocal}. We shall use 
properties of the operator of restriction~$L^\eps$ defined in~(\ref{def:Leps}) and of the projector~$\Pi_n(\eps D_x)$. Then, we prove  \textit{a priori} estimates for solutions of equation~\eqref{eq:U} and use them to reduce the dynamics of our original problem to those of equation \eqref{eq:U_components} (Corollary~\ref{prop:adiab1}).

\medskip

Note that, modulo adding a positive constant to equation \eqref{eq:schro}, we may assume that $P(\eps D_x)$ is a non-negative operator. With this in mind, the following estimates, that will be repeatedly used in what follows, hold.

\begin{remark}\label{rem:eqnorm}
There exists a constant $c>0$ such that:
\[
c^{-1}\|U\|_{H^s_\eps(\R^d\times\T^d)}\leq\
\|\left\langle\eps D_x\right\rangle^s U\|_{L^2(\R^d\times\T^d)}+\|P(\eps D_x)^{s/2}U\|_{L^2(\R^d\times\T^d)}
\leq
 c\|U\|_{H^s_\eps(\R^d\times\T^d)},
 \]
for every $U\in L^2(\R^d\times\T^d)$ and $\eps>0$, where, as usual,
$\left\langle\xi\right\rangle=(1+|\xi|^2)^{1/2}$ and where the sets $H^s_\eps$ have been defined in~(\ref{def:Hseps}).
\end{remark}

\subsection{High frequency behavior of the operator of restriction to the diagonal and of the Bloch projectors}\
We first focus on the properties of the operator of restriction to the diagonal~$L^\eps$ and prove its boundedness in appropriate functional spaces.   

\begin{lemma} \label{lem:eoscxy} Suppose $s>d/2$, then the operator 
$$L^\eps: L^2(\R^d_x;H^s(\T^d_y ))\To L^2(\R^d)$$ 
is uniformly bounded in $\eps$. \\
Moreover, if $U^\eps\in L^2(\R^d_x;H^s(\T^d_y ))$ satisfies the estimate:
\begin{equation}\label{eq:epsocsxy}
\limsup_{\eps\to 0^+}\|\mathds{1}_R(\eps D_x)U^\eps\|_{L^2(\R^d;H^s(\T^d ))}\Tend{R}{\infty}0,
\end{equation}
where $\mathds{1}_R$ is the characteristic function of $\{|\xi|>R\}$, 
then the sequence $(L^\eps U^\eps)$ is bounded in $L^2(\R^d)$ and $\eps$-oscillating.
\end{lemma}

\begin{remark}\label{rem:psieps0epsosc}
Suppose that $(U^\eps)$ is bounded in  $H^r_\eps(\R^d\times\T^d)$ for some $r>d/2$. Then condition~\eqref{eq:epsocsxy} is satisfied for every $d/2<s<r$. This follows from the bound:
\[
\|\mathds{1}_R(\eps D_x)U^\eps\|_{L^2(\R^d;H^s(\T^d ))}\leq R^{s-r}\|U^\eps\|_{H^r_\eps(\R^d\times\T^d)}.
\] 
In particular, if $\psi^\eps_0$ satisfies item (3) of Assumptions~\ref{hypothesis:main}, then $(\psi^\eps_0)$ is $\eps$-oscillating. 
\end{remark}

\begin{proof}
Let $U^\eps \in L^2(\R^d_x;H^s(\T^d_y ))$ and write 
$$U^\eps(x,y)=\sum_{k\in\Z^d}U_k^\eps(x){\rm e}^{i2\pi k\cdot y},$$
and
\[
\|U{^\eps}\|_{L^2(\R^d_x;H^s(\T^d_y ))}^2=\sum_{k\in\Z^d}\left\langle k \right\rangle^{{2s}} \|U_k{^\eps}\|^2_{L^2(\R^d)}.
\]
Then there exist constants $C,C_{d,s}>0$ such that 
$$\sum_{k\in\Z^d}\|U_k^\eps\|_{L^2(\R^d)}\leq C\left(\sum_{k\in\Z^d}|k|^{2s} \|U_k^\eps\|_{L^2(\R^d)}^2\right)^{1/2} \leq C_{d,s} \|U^\eps\|_{L^2(\R^d_x;H^s(\T^d_y ))},$$
and therefore:
\begin{equation}\label{eq:bdVe}
\|L^\eps U^\eps\|_{L^2(\R^d)}\leq \sum_{k\in\Z^d}\|U_k^\eps\|_{L^2(\R^d)}\leq C_{d,s}\|U^\eps\|_{L^2(\R^d_x;H^s(\T^d_y ))}.
\end{equation}
Let us now show that, under the hypothesis of the proposition, $v^\eps:=L^\eps U^\eps$ defines an $\eps$-oscillating sequence. 
Given $\delta>0$, since $s>d/2$, there exists $N_\delta>0$ such that
$$\sum_{|k|>N_\delta}|k|^{-2s}<\delta^2.$$
Define:
$$v^\eps_\delta(x)=\sum_{|k|\leq N_\delta}U_k^\eps(x){\rm e}^{i2\pi k\cdot \frac{x}{\eps}}.$$
Clearly,
$$\|v^\eps-v^\eps_\delta\|_{L^2(\R^d)}\leq \delta  \|U^\eps\|_{L^2(\R^d_x;H^s(\T^d_y ))}.$$
Therefore, it suffices to show that for any $\delta>0$ the sequence $(v^\eps_\delta)$ is $\eps$-oscillating.
The Fourier transform of $v^\eps_\delta$ is:
$$\widehat{v^\eps_\delta}(\xi)=\sum_{|k|\leq N_\delta}\widehat{U_k^\eps}\left( \xi -\frac{2\pi k}{\eps}\right).$$
Therefore, 
$$\|\mathds{1}_R(\eps D_x)v^\eps_{\delta}\|_{L^2(\R^d)}\leq \sum_{|k|\leq N_\delta}\|\mathds{1}_R(\eps D_x+2\pi k)U_k^\eps\|_{L^2(\R^d)}.$$
If $R>R_0$ for $R_0>0$ large enough, one has $\mathds{1}_R(\cdot+2\pi k)\leq \mathds{1}_{R/2}$ for every $|k|\leq N_{\delta}$. This allows us to conclude that for $R>R_0$:
$$\|\mathds{1}_R(\eps D_x)v^\eps_\delta\|_{L^2(\R^d)}\leq \sum_{|k|\leq N_\delta}\|\mathds{1}_{R/2}(\eps D_x)U_k^\eps\|_{L^2(\R^d)}\leq C_{d,s}\|\mathds{1}_R(\eps D_x)U^\eps\|_{L^2(\R^d;H^s(\T^d ))}$$
and the conclusion follows.
\end{proof}

We shall also need information on the derivatives with respect to $\xi$ of the operator $\Pi_n(\xi)$. We recall the formula  
$$\Pi_n(\xi) =- {1\over 2i\pi} \sum_{j=1}^N\chi_j(\xi)\oint_{\mathcal C_j} (P(\xi)-z)^{-1} dz$$
where the functions $\chi_j\in\mathcal{C}^\infty(\R^d/2\pi\Z^d)$ form a partition of unity and, for $j=1,..N$, ${\mathcal C_j}$ is a contour in the complex plane separating $\varrho_n(\xi)$, for $\xi\in\supp\chi_j$, form the remainder of the spectrum. The existence of such contours is guaranteed by the fact that $\varrho_n(\xi)$ is of constant multiplicity for all $\xi\in\R^d$ and, thus, is separated from the remainder of the spectrum. As a consequence of this formula, of Lemma~\ref{rem:eqnorm} and of the relation
\[
[\Pi_n(\eps D_x),P(\eps D_x)^{s/2}]=[\Pi_n(\eps D_x),\left\langle \eps D_x\right\rangle^s]=0,
\] we deduce the following result.

\begin{lemma}\label{boundednessdPi}
The map $\xi\mapsto \Pi_n(\xi)$ is a smooth bounded map from $\R^d$ into ${\mathcal L} (L^2(\T^d))$. In addition, the operator $\Pi_n(\eps D_x)$ maps the space $H^s_\eps(\R^d\times\T^d)$ into itself.
\end{lemma}

\subsection{A priori estimates on $U^\eps(t,\cdot)$ }

In order to derive the desired properties of $\psi^\eps_n(t,x)$, the solution to \eqref{eq:U_components}, we need to prove some  \textit{a priori} estimates for the solutions of equation~(\ref{eq:U}). We will use them for reducing the analysis of $\psi^\eps(t,\cdot)$ (the solution to our original problem \eqref{eq:schro}) to that of $\psi^\eps_n(t,\cdot)$.

\begin{lemma}\label{lem:wp}
Given $s\geq 0$, there exists a constant $C_s>0$ such that any solution $U^\eps$ to~(\ref{eq:U}) with initial datum $U^\eps_0\in H^s(\R^d\times\T^d)$ satisfies:
\begin{equation}\label{eq:wpxy}
\|U^\eps(t,\cdot)\|_{H^s_\eps(\R^d\times\T^d)}\leq \|U^\eps_0\|_{H^s_\eps(\R^d\times\T^d)}+C_s\eps|t|,
\end{equation}
uniformly in $\eps >0$. 
\end{lemma}

\begin{corollary}\label{cor:psiepsosc}
Lemma~\ref{lem:wp} and Remark~\ref{rem:psieps0epsosc} imply that for all $t\in\R$, the family $(\psi^\eps(t,\cdot))$ is $\eps$-oscillating. 
\end{corollary}

\begin{proof}
In view of Remark~\ref{rem:eqnorm},
we are first going to study the  families $$(\left\langle\eps D_x\right\rangle U^\eps)\;\;{\rm and}\;\;(P(\eps D_x)^{1/2}U^\eps).$$  Start noticing that $\left\langle\eps D_x\right\rangle U^\eps$ satisfies the equation 
\begin{equation}\label{eq:derU}
i\eps^2 \partial_t (\left\langle\eps D_x\right\rangle U^\eps) = P(\eps D_x) (\left\langle\eps D_x\right\rangle U^\eps) + \eps^2 V_{\rm ext} \left\langle\eps D_x\right\rangle U^\eps - \eps^2 [V_{\rm ext}, \left\langle\eps D_x\right\rangle] U^\eps.
\end{equation}
As a consequence, using the boundedness of $\nabla_x V_{\rm ext}$ on $\R\times\R^d$, we obtain by the symbolic calculus of semiclassical pseudodifferential operators, that the source term can be estimated by:
\[
\|[V_{\rm ext}(t,\cdot), \left\langle\eps D_x\right\rangle] U^\eps(t,\cdot)\|_{L^2(\R ^d \times\T^d)}\leq  C\eps \|U^\eps(t,\cdot)\|_{L^2(\R ^d \times\T^d)},
\]
for some constant $C>$ independent of $\eps>0$ and $t\in\R$.
Using standard energy estimates, we deduce the existence of a constant $C_1>0$ such that for all $t\in\R$,
$$\| \left\langle\eps D_x\right\rangle U^\eps(t,\cdot)\| _{L^2(\R ^d \times\T^d)}\leq \| \left\langle\eps D_x\right\rangle U^\eps_0\| _{L^2(\R ^d \times\T^d)}+ C_1\eps |t|.$$
A completely analogous argument yields the estimate:
\[
\| P(\eps D_x)^{1/2} U^\eps(t,\cdot)\| _{L^2(\R ^d \times\T^d)}\leq \| P(\eps D_x)^{1/2} U^\eps_0\| _{L^2(\R ^d \times\T^d)}+ C_1\eps |t|.
\]
A standard recursive argument gives, for all $s\in \N$, the existence of a constant $C_s>0$ such that for all $t\in\R$,
\begin{multline*}
\| \left\langle\eps D_x\right\rangle^s U^\eps(t,\cdot)\| _{L^2(\R ^d \times\T^d)}+\| P(\eps D_x)^{s/2} U^\eps(t,\cdot)\| _{L^2(\R ^d \times\T^d)}\leq \\
\| \left\langle\eps D_x\right\rangle^s U^\eps_0\| _{L^2(\R ^d \times\T^d)}+\| P(\eps D_x)^{s/2} U^\eps_0\| _{L^2(\R ^d \times\T^d)}+ C_s\eps |t|,
\end{multline*}
and the result follows for any $s\in\R^+$ by interpolation. 
\end{proof}

We now focus on the case where the initial data $U^\eps_0$ belongs to a particular Bloch eigenspace: $U^\eps_0= \Pi_n(\eps D_x) U^\eps_0$. We set
$$\widetilde U^\eps(t,\cdot)= \Pi_n(\eps D_x) U^\eps(t,\cdot).$$
Note that by Lemma~\ref{boundednessdPi}, for any $t\in\R$, the family $\widetilde U^\eps (t,\cdot)$ is uniformly bounded in $H^s_\eps(\R^d\times \T^d)$. 

\begin{lemma}\label{lem:Uadiab}
Assume $U^\eps_0= \Pi_n(\eps D_x) U^\eps_0$ and consider $\widetilde U^\eps(t,\cdot)$ as defined above.
 Then,  for all $T>0$, there exists $C_{T}>0$ such that 
$$\sup_{t\in[0,T]} \left\| U^\eps(t,\cdot) -\widetilde U^\eps(t,\cdot) \right\|_{H^s_\eps(\T^d \times \R^d)} \leq C_{T}\eps.$$
\end{lemma}

Let us prove now Lemma~\ref{lem:Uadiab}.

\begin{proof}
 Note first that, in view of Remark \ref{rem:eqnorm}, it is enough to prove the uniform boundedness in $L^2(\T^d\times\R^d)$ of 
$$ U^\eps(t,\cdot) -\widetilde U^\eps(t,\cdot)  ,\;\;
  P(\eps D_x)^{s/2}(U^\eps(t,\cdot) -\widetilde U^\eps(t,\cdot) )\;\;{\rm and}\;\;
  \langle \eps D_x\rangle ^s( U^\eps(t,\cdot) -\widetilde U^\eps(t,\cdot)).$$
We have $U^\eps(0,\cdot)=\widetilde U^\eps(0,\cdot)$ and $\widetilde U^\eps$ solves 
\begin{equation}\label{eq:Utilde}
 i\eps^2\partial_t \widetilde U^\eps(t,x)=P(\eps D_x) \widetilde U^\eps (t,x)+ \eps^2 V_{\mathrm{ext}}(t,x)\widetilde U^\eps (t,x)+ \eps^2  B^\eps (t)U^\eps(t,x), 
\end{equation}
with 
$$B^\eps(t) =[ \Pi_n(\eps D_x), V_{\mathrm{ext}}(t,\cdot) ].$$
The symbolic calculus of semiclassical pseudodifferential operators implies that:
\[
\|B^\eps(t) U^\eps(t,\cdot)\|_{L^2(\R^d\times\T^d)}=O(\eps), \quad \text{ locally uniformly in  }t.
\]
As for $\langle \eps D_x\rangle  \widetilde U^\eps$ one has:
\[
 i\eps^2\partial_t (\langle \eps D_x\rangle\widetilde U^\eps)=P(\eps D_x) \langle \eps D_x\rangle\widetilde U^\eps + \eps^2 V_{\mathrm{ext}} \langle \eps D_x\rangle \widetilde U^\eps +  \eps^2C^\eps\langle \eps D_x\rangle U^\eps- \eps^2[V_{\rm ext}, \left\langle\eps D_x\right\rangle] \widetilde U^\eps, 
\]
with, 
\[
C^\eps=[ \Pi_n(\eps D_x), \langle \eps D_x\rangle V_{\mathrm{ext}}\langle \eps D_x\rangle^{-1}].
\]
Again, the symbolic calculus gives that $\|C^\eps(t)\langle \eps D_x\rangle U^\eps(t,\cdot)\|_{L^2(\R^d\times\T^d)}=O(\eps)$ locally uniformly in $t$.
Taking into account that $\langle \eps D_x\rangle U^\eps$ satisfies equation \eqref{eq:derU} and is bounded in $L^2(\R^d\times\T^d)$, one concludes that:
\[
  \|\langle \eps D_x\rangle ( U^\eps(t,\cdot) -\widetilde U^\eps(t,\cdot))\|_{L^2(\R^d\times \T^d)}\leq C\eps|t|.
\]
An analogous reasoning holds for $P(\eps D_x)^{1/2}(U^\eps(t,\cdot) -\widetilde U^\eps(t,\cdot) )$. One concludes using an inductive argument following the lines of the end of the proof of Lemma \ref{lem:wp}.
\end{proof}

\subsection{Analysis of the Bloch component $\psi^\eps_n$}\label{sec:psiepsn}

By the definition of $\psi^\eps_n(t,x)$, we have 
$$\psi^\eps_n(t,\cdot)= L^\eps \widetilde U^\eps(t,\cdot);$$
and the family is bounded in $L^2(\R^d)$ for all $t\in\R$.
Moreover, as a corollary of Lemma~\ref{lem:Uadiab}, the following holds. 
\begin{corollary}\label{prop:adiab1}
Suppose that $\psi^\eps$ and $\psi^\eps_n$ are the respective solutions of equations  \eqref{eq:schro} and \eqref{eq:U_components} with the same initial datum $L^\eps  U^\eps_0$, where $U^\eps_0= \Pi_n(\eps D_x) U^\eps_0$. Then for every $T>0$ there exist $C_T>0$ such that, uniformly in $\eps$,
\begin{equation*}
\sup_{t\in[0,T]} \| \psi^\eps(t,\cdot)- \psi^\eps_n(t,\cdot)\|_{L^2(\R^d)} \leq C_{T}\eps.
\end{equation*}
\end{corollary}
The proof is a direct consequence of Lemma~\ref{lem:Uadiab}, since Lemma~\ref{lem:eoscxy} ensures that
$$\| \psi^\eps(t,\cdot)- \psi^\eps_n(t,\cdot)\|_{L^2(\R^d)} \leq C \| U^\eps(t,\cdot)-\widetilde U^\eps(t,\cdot)\|_{L^2(\R^d,H^s(\T^d))}.$$

We now conclude our analysis of the Bloch component $\psi^\eps_n(t,\cdot)$. The following result gathers the remaining information that we will need in order to conclude, together with Corollary~\ref{prop:adiab1}, the proof of Theorem~\ref{mainresult}.

\begin{proposition}
\label{lem:decompsolution}
The family $\psi_n^\eps$ solves equation~(\ref{eq:U_components}) 
$$
\left\{ \begin{array}{l}
i\eps^2 \partial_{t} \psi^\eps_{n} (t,x)- \varrho_{n}(\eps D_x) \psi^\eps_{n}(t,x)- \eps^2 V_{\rm ext}(t,x) \psi^\eps_{n}(t,x)= \eps^2 f^\eps_n(t,x),  \vspace{0.2cm}\\
\psi^\eps_{n}|_{t=0}(x)= \psi_{0}^{\eps}(x)
\end{array}\right.
$$
 with~(\ref{eq:fneps}):
 $\|f^\eps_n(t,\cdot)\|_{L^2(\R^d)}\leq C\eps$ for all  $t\in\R,\,\eps>0.$
\end{proposition}

\begin{proof} Let us first prove that $\psi_n^\eps$ solves~(\ref{eq:U_components}). We denote by $J$ the set of the indexes of the Bloch eigenfunctions $\varphi_j(\cdot,\xi)$ which form an orthonormal basis of ${\rm Ran}\, \Pi_n(\xi)$. 
Define for $j\in J$,
 $$u^\eps_j(t,x):=\int_{\T^d} \overline \varphi_j(y,\eps D_x) \widetilde U^\eps(t,x,y) dy,$$
and notice that:
\[
\psi^\eps_n(t,x)=(L^\eps \widetilde U^\eps)(t,x)=\sum_{j\in J}  \varphi_j\left({x\over\eps},\eps D_x\right) u^\eps_j(t,x).
\] 
Since $\widetilde U^\eps$ solves \eqref{eq:Utilde} and $P(\xi) \varphi_j(\cdot,\xi) = \varrho_n(\xi) \varphi_j(\cdot,\xi)$ for all $\xi\in\R^d$, 
the family~$u^\eps_j$ solves:
$$i\eps^2\partial_t u^\eps_j (t,x)= \varrho_n(\eps D_x) u^\eps_j(t,x)+ \eps ^2 V_{\rm ext}(t,x) u^\eps_j(t,x)+\eps^2 g_j^\eps(t,x),$$
where:
$$g^\eps_j(t,x):=\int_{\T^d}[\overline{\varphi_j}(y,\eps D_x),V_{\rm ext}(t,x)]U^\eps (t,x,y)dy.$$
Since $\varrho_n(\xi)$ is $2\pi\Z^d$-periodic, it is easy to check that:
$$[L^\eps\varphi_j(\cdot,\eps D_x),\varrho_n(\eps D_x)]=0.$$
Summing the relations over $j\in J$, this implies~(\ref{eq:U_components}) 
with 
$
f^\eps_n  = L^\eps [\Pi_n(\eps D_x), V_{\rm ext}  ] U^\eps.
$
Now, Lemma \ref{lem:eoscxy} and the symbolic calculus of pseudodifferential operators gives, for any $t\in\R$:
$$\|  f^\eps_n(t,\cdot)\|_{L^2(\R^d)} \leq C\, \left\| [\Pi_n(\eps D_x), V_{\rm ext} (t,\cdot) ] U^\eps(t,\cdot)\right\|_{L^2(\R^d;H^s(\T^d))}\leq C'\eps \|U^\eps(t,\cdot)\|_{L^2(\R^d;H^s(\T^d))},$$
which concludes the proof.
\end{proof}

\subsection{Proofs of Theorems~\ref{mainresult} }

The proof of Theorem~\ref{mainresult} (which implies Corollary~\ref{mainresul:case1}) easily follows from our results so far.

\begin{proof} By Corollary~\ref{cor:psiepsosc}, the family $(\psi^\eps(t,\cdot))$ is $\eps$-oscillating. Therefore, the weak limits of $|\psi^\eps(t,x)|^2dx$ are the projection on $\R^d_x$ of the Wigner measures associated with $(\psi^\eps(t,\cdot))$. By Corollary~\ref{prop:adiab1}, the Wigner measures of $(\psi^\eps(t,\cdot))$ coincide with those of $(\psi^\eps_n(t,\cdot))$. Finally, 
Proposition~\ref{lem:decompsolution}
allows us to use the results of Theorem~\ref{theo:disc} for determining the Wigner measure of $(\psi^\eps_n(t,\cdot))$. 
\end{proof}

\subsection{Some comments on initial data that are a finite superposition of Bloch modes}\label{sec:superposition}
Our results also apply to initial data that are a finite linear combination of the form:
\begin{equation}\label{eq:datasuperposition}
\psi^\eps_0=\sum_{n\in{\mathcal N}} L^\eps U^\eps_{0,n} 
\end{equation}
with ${\mathcal N}$ a finite subset of $\N$ such that for all $n\in{\mathcal N}$, 
$P(\eps D_x) U^\eps_{0,n}= \varrho_n(\eps D_x) U^\eps_{0,n},$
for distinct $\varrho_n$ of constant multiplicity and $U^\eps_{0,n}$ uniformly bounded in $H^s_{\eps}(\R^d\times \T^d)$ for all $n\in{\mathcal N}$. 

\begin{proposition}\label{theo:superposition}
Assume we turn $(3)$ into~(\ref{eq:datasuperposition}) in the hypothesis of Assumption~\ref{hypothesis:main}  and that item (2) of Assumption~\ref{hypothesis:main} holds for every $\varrho_n$ with $n\in{\mathcal N}$.  Then, 
there exist a subsequence $(\eps_k)_{k\in\N}$, positive measures $\nu_n\in\mathcal{M}_+(T^*\Lambda_n)$, and  measurable families of self-adjoint, positive, trace-class operators
$$M_{0,n}:T^*_\xi\Lambda_n\ni (v,\xi)\longmapsto M_{0,n}(v,\xi)\in \mathcal{L}_+^1(L^2(N_\xi\Lambda_n)),\quad \Tr_{L^2(N_\xi\Lambda_n)} M_{0,n}(v,\xi)=1,$$ 
such that for
for every $a<b$ and $\phi\in \mathcal{C}_c(\R^d)$ one has:
$$\displaylines{
\lim_{k\to\infty}\int_a^b\int_{\R^d}\phi(x)|\psi^{\eps_k}(t,x)|^2dxdt\hfill\cr\hfill =\sum_{n\in{\mathcal N}}\int_a^b \int_{T^*\Lambda_n}\Tr_{L^2(N_\xi\Lambda_n)}\left[m_\phi(v,\xi)M_n(t,v,\xi)\right]\nu_j(dv,d\xi)dt,
\cr}$$
where $M_n(\cdot, v,\xi)\in \mathcal{C}(\R;\mathcal{L}_+^1(L^2(N_\xi\Lambda_n))$ solves the Heisenberg equation~(\ref{eq:heis}) with initial data $M_{0,n}$ associated with the concentration of $\psi^\eps_0$ on $\Lambda_n$.
\end{proposition}

\begin{proof} 
We  associate to any $n\in \N$ their respective Bloch components $\psi^\eps_n(t,\cdot)$ of $\psi^\eps (t,\cdot)$ as we previously did. We juste have to prove that
for all $n,n'\in{\mathcal N}$, $n\not=n'$,
$$\int_\R\theta(t) \left({\rm op}_\eps(a) \psi^\eps_n(t,\cdot),\psi^\eps_{n'}(t,\cdot)\right)\Tend{\eps}{0} 0,$$
which implies that 
 the Wigner measure of $\sum_{n\in{\mathcal N}}\psi^\eps_n$ is the sum of the Wigner measures of the~$\psi^\eps_n$. 
We take 
$a\in{\mathcal C}_c^\infty (\R^{2d})$  and  $\tilde{a} = (\varrho_n-\varrho_{n'})^{-1} a \in {\mathcal C}_c^\infty\left(\R^{2d}\right)$; then, for all $t\in\R$,
\begin{align*}
\left({\rm op}_\eps(a) \psi^\eps_n(t,\cdot),\psi^\eps_{n'}(t,\cdot)\right) &= \left({\rm op}_\eps(\tilde{a})\varrho_n(\eps D_x) \psi^\eps_n(t,\cdot),\psi^\eps_{n'}(t,\cdot)\right) \\&- \left({\rm op}_\eps(\tilde{a}) \psi^\eps_n(t,\cdot),\varrho_{n'}(\eps D_x)\psi ^\eps_{n'}(t,\cdot)\right) + O(\eps)
\end{align*}
from which we deduce:
$$\left({\rm op}_\eps(a) \psi^\eps_n(t,\cdot),\psi^\eps_{n'}(t,\cdot)\right)= i\eps^2 {d\over dt} \left({\rm op}_\eps(\tilde{a}) \psi^\eps_n(t,\cdot),\psi^\eps_{n'}(t,\cdot)\right) +O(\eps).$$
Therefore, if $\theta\in {\mathcal C}_c^\infty(\R)$,
then 
$$\displaylines{
 \int_\R \theta(t) \left({\rm op}_\eps(a) \psi^\eps_n(t,\cdot),\psi^\eps_{n'}(t,\cdot)\right)dt =  O(\eps ) + i\eps^2 \int _\R \theta(t) {d\over dt} \left({\rm op}_\eps(\tilde{a}) \psi^\eps_n(t,\cdot),\psi^\eps_{n'}(t,\cdot)\right) dt\hfill\cr\hfill
= O(\eps ) - i\eps^2 \int _\R \theta'(t) \left({\rm op}_\eps(\tilde{a})\psi^\eps_n(t,\cdot),\psi^\eps_{n'}(t,\cdot)\right)dt = O(\eps).\qquad \cr}$$
\end{proof}


\appendix 
\section{Semiclassical pseudodifferential operators}\label{sec:pdo}

In this appendix we recall a few basic notions on the theory of pseudodifferential operators that we use trough this article. The reader can consult the references \cite{AlinhacGerard, DimassiSjostrand, F:14, MartinezBook, Zwobook} for additional background and for proofs of the results that follow.

\medskip 

Recall that given a function $a\in\mathcal{C}^\infty(\R^d\times\R^d)$ that is bounded together with its derivatives
 (we denote the space of all such functions by ${S}$), one defines the semiclassical pseudodifferential operator of symbol $a$ obtained through the Weyl quantization rule to be the operator $\op_\eps(a)$ that acts on functions $f\in\mathcal{S}(\R^d)$ by: 
$$\op_\eps(a) f(x)=\int_{\R^{d}\times\R^d} a\left(\frac{x+y}{2},\eps\xi\right) {\rm e}^{i \xi\cdot (x-y)} f(y)dy\,\frac{d\xi}{(2\pi)^d}.$$  
These operators are bounded in $L^2(\R^d)$. The Calderón-Vaillancourt theorem \cite{CV} ensures the existence of a constant $C_d>0$ such that for every $ a\in S$ one has
\begin{equation}\label{est:pseudo}
\| \op_\eps(a)\|_{{\mathcal L}(L^2(\R^d))}\leq C_d\, 
N(a),
\end{equation}
where
$$
N_d(a):=\sum_{\alpha\in\N^{2d},|\alpha|\leq J_0} \sup_{\R^d\times\R^d}|\partial_{x,\xi}^\alpha a|
$$
for some $J_0\in\N$ depending only on~$d$. 
We make use repeatedly of the following result, known as the symbolic calculus for pseudodifferential operators.

\begin{proposition}\label{prop:symbol}
Let $a,b\in S$, then
$$\op_\eps(a)\op_\eps(b)  = \op_\eps(ab)+\frac{\eps}{2i} \op_\eps(\{a,b\})+\eps^2 R^{(2)}_\eps,$$
with $\{a,b\}=\nabla_\xi a \cdot \nabla _x b-\nabla _xa\cdot \nabla_\xi b$ and 
$$\left[\op_\eps(a),\op_\eps(b)\right] ={\eps\over i}\op_\eps(\{a,b\}) +\eps^3 R_\eps^{(3)},$$
$$\| R^{(j)}_\eps\|_{{\mathcal L}(L^2(\R^d))}\leq C \,\sup_{|\alpha|+|\beta|=j} N_d(\partial_\xi^\alpha \partial_x^{\beta}  a) N_d(
 \partial_\xi^\beta \partial_x^{\alpha}  b),\quad j=1,2,$$
for some constant $C>0$ independent of $a$, $b$ and $\eps$.
\end{proposition}

\section{Trace operator-valued measures}\label{sec:ovm}
In this appendix we recall general considerations on operator-valued measures. 
Let $X$ be a complete metric space and $(Y,\sigma)$ a measure space; write $\mathcal{H}:=L^2(Y,\sigma)$ and denote by $\mathcal{L}^1(\mathcal{H})$, $\mathcal{K}(\mathcal{H})$ and $\mathcal{L}(\mathcal{H})$ the spaces of trace-class, compact and bounded operators on $\mathcal{H}$ respectively. A trace-operator valued Radon measure on $X$ is a linear functional:
$$M:\mathcal{C}_0(X)\To \mathcal{L}^1(\mathcal{H})$$
satisfying the following boundedness condition. For every compact $K\subset X$ there exist a constant $C_K>0$ such that:
$$\Tr|M(\phi)|\leq C_K \sup_K|\phi|,\quad \forall \phi\in\mathcal{C}_0(K).$$
Such an operator-valued measure is positive if for every $\phi\geq0$, $M(\phi)$ is an Hermitian positive operator.
Let $M$ be a positive trace operator-valued measure on $X$, denote by $\nu\in\mathcal{M}_+(X)$ the positive real measure defined by:
$$\int_X\phi(x)\nu(dx)=\Tr M(\phi),\quad \forall \phi\in\mathcal{C}_0(X).$$
The Radon-Nikodym theorem for operator valued measures (see, for instance, the appendix in \cite{GerardMDM91}) ensures the existence of a $\nu$-locally integrable function:
$$Q:X\longmapsto \mathcal{L}^1(\mathcal{H}),\quad \Tr Q(x)=1, \quad Q(x) \text{ positive Hermitian for } \nu \text{-a.e.} x\in X,$$ 
such that: 
$$M(\phi)=\int_X \phi(x) Q(x)\nu(dx),\quad \forall \phi\in\mathcal{C}_0(X).$$
Note that this formula implies that $M$ can be identified to a positive element of the dual of $\mathcal{C}_0(X;\mathcal{K}(\mathcal{H}))$ via:
$$\langle M, T \rangle \equiv \int_X \Tr[T(x)M(dx)]:=\int_X \Tr(T(x)Q(x))\nu(dx),\quad T\in \mathcal{C}_0(X;\mathcal{K}(\mathcal{H})).$$ 
It can be also shown that every such positive functional arises in this way.
Let  $(e_j(x))_{j\in\N}$ denote an orthonormal basis of $\mathcal{H}$ consisting of eigenfunctions of $Q(x)$:
$$Q(x)e_j(x)=\varrho_j(x) e_j(x),\quad \sum_{j=1}^\infty \varrho_j(x)=1,\quad \nu\text{-a.e.}.$$ 
Clearly, both $\varrho_j$ and $e_j$, $j\in\N$, are locally $\nu$-integrable and
$$Q(x)=\sum_{j=1}^\infty \varrho_j(x)|e_j(x)\rangle \langle e_j(x)|,\quad \nu\text{-a.e.},$$
where, as usual, $|e_j(x)\rangle \langle e_j(x)|$ denotes the orthogonal projection in $\mathcal{H}$ onto $e_j(x)$. Moreover, as a consequence of the monotone convergence theorem, the following result easily follows.
\begin{lemma}\label{lem:abs}
Let $M$ be a positive trace operator-valued measure on $X$. Then there exist  a non-negative function $\rho\in L^1_{\loc}(X,\nu;L^1(Y,\sigma))$ such that, for every $a\in \mathcal{C}_0(X; L^\infty(Y,\sigma))$ one has:
$$\int_X \Tr[m_a(x) M(dx)] = \int_X \int_Y a(x,y) \rho(x,y)\sigma(dy) \nu(dx),$$
where $m_a(x)$ denotes the operator acting on $\mathcal{H}$ by multiplication by $a(x,\cdot)$. The density $\rho$ is given by:
$$\rho(x,y)=\sum_{j=1}^\infty \varrho_j(x)|e_j(x,y)|^2.$$
\end{lemma} 


\section{Proof of~Lemma~\ref{lem:geometry}}\label{sec:app_proof}

We denote by ${\mathcal F}_\eps$ the semi-classical Fourier transform defined for $f\in L^2(\R^d)$ by
$${\mathcal F}_\eps f(\xi)=(2\pi\eps)^{-d/2} \widehat f\left({\xi\over\eps}\right)$$
and we observe that for $a\in{\mathcal C}_c^\infty(\R^{3d})$, 
$$\left(\op_\eps(a_\eps)f\;,\;f\right)=(2\pi\eps)^{-d} \int_{\R^{3d}} a_\eps\left(-x,{\xi+\xi'\over 2}\right) {\rm e}^{{i\over \eps} x\cdot (\xi-\xi')}{\mathcal F}_\eps f(\xi')\overline{{\mathcal F}_\eps f}(\xi) d\xi\,d\xi'\,dx,$$
where $a_\eps$ is associated with $a$ according to~(\ref{def:aeps}). 
 We consider a smooth cut-off function $\chi$ which is equal to $1$ on the support of $a$  so that we have
 $a(x,\xi) \chi(\xi)=a(x,\xi)$ and we write  
$$\displaylines{ \left(\op_\eps(a_\eps)f\;,\;f\right) =\hfill \cr
(2\pi\eps)^{-d} \int_{\R^{3d}} a_\eps\left(-x,{\xi+\xi'\over 2}\right) {\rm e}^{{i\over \eps} x\cdot (\xi-\xi')}{\mathcal F}_\eps f(\xi')\overline{{\mathcal F}_\eps f}(\xi) \chi(\xi)\chi(\xi')d\xi\,d\xi'\,dx + O(\eps).
\cr}$$ 
The rest term $O(\eps)$ comes from Taylor formula close to ${\xi+\xi'\over 2}$, the observation that 
$$(\xi_j-\xi'_j) {\rm e}^{{i\over \eps} x\cdot (\xi-\xi')}={ \eps\over i} \partial_{x_j} \left( {\rm e}^{{i\over \eps} x\cdot (\xi-\xi')}\right),\;\;1\leq j\leq d,$$
and the use of integration by parts {in $x$}. Similarly, we just need  to consider vectors~$(\xi,\xi')$ which are close to the diagonal and if we introduce a smooth function~$\Theta$ compactly supported on $|\xi|\leq 1$ and equal to $1$ close to $0$, then for some $\delta>0$ (that will be chosen small enough later), we have
$$\displaylines{ \left(\op_\eps(a_\eps)f\;,\;f\right) =
(2\pi\eps)^{-d} \int_{\R^{3d}} a_\eps\left(-x,{\xi+\xi'\over 2}\right) \hfill \cr\hfill
\times\, {\rm e}^{{i\over \eps} x\cdot (\xi-\xi')}{\mathcal F}_\eps f(\xi')\overline{{\mathcal F}_\eps f}(\xi) \Theta\left({\xi-\xi'\over\delta}\right)\chi(\xi)\chi(\xi')d\xi\,d\xi'\,dx + O(\eps).
\cr}$$ 
We are left with the integral 
  \begin{eqnarray*}
I_\eps &=& (2\pi\eps)^{-d} \int_{\R^{3d}} a_\eps\left(-x,{\xi+\xi'\over 2}\right) {\rm e}^{{i\over \eps} x\cdot (\xi-\xi')}{\mathcal F}_\eps f(\xi')\overline{{\mathcal F}_\eps f}(\xi) \chi(\xi)\chi(\xi')\\
& & \qquad \times\, \Theta\left({\xi-\xi'\over\delta}\right)d\xi\,d\xi'\,dx\\
 &=&   (2\pi\eps)^{-d} \int_{\R^{3d}} a_\eps\left(-x,{\Phi(\zeta)+\Phi(\zeta')\over 2}\right) {\rm e}^{{i\over \eps} x\cdot (\Phi(\zeta)-\Phi(\zeta'))}{\mathcal F}_\eps f(\Phi(\zeta'))\\
 &  &\qquad \times \,\overline{{\mathcal F}_\eps f}(\Phi(\zeta)) J_\Phi(\zeta)\,  J_\Phi(\zeta')\, \chi\circ \Phi(\zeta)\, \chi\circ\Phi(\zeta') \Theta\left({\Phi(\zeta)-\Phi(\zeta')\over\delta}\right)d\zeta\,d\zeta'\,dx
 \end{eqnarray*}
 where $\zeta\mapsto J_\Phi(\zeta)$ is the Jacobian of the diffemorphism~$\Phi$. 
 Setting 
 $$\zeta=\theta+\eps {v\over 2}\;\;{\rm and}\;\;{\zeta'}=\theta-\eps {v\over 2},$$
 we have for $t\in\R$,
$$ \Phi(\theta+\eps tv) =\Phi(\theta) +\eps t d\Phi(\theta) v +\eps^2 \int_0^1 d^2\Phi(\theta+\eps tsv) [v,v] (1-s) ds,$$
 whence 
 $$\displaylines{
 {1\over 2} \left( \Phi(\zeta)+ \Phi(\zeta')\right) =\Phi(\theta) +{\eps^2\over 2} B_\eps^+(\theta,v)[v,v],\cr
  \Phi(\zeta)- \Phi(\zeta')=\eps d\Phi(\theta) v +\eps^2 B_\eps^-(\theta,v)[v,v],\cr
 }$$
 with 
 $$ B_\eps^\pm(\theta,v)=\int_0^1 d^2 \left(\Phi \left(\theta+\eps s {v\over 2}\right)\pm\Phi \left(\theta-\eps s {v\over 2}\right)\right)(1-s)ds.$$
 Note that  the functions $B_\eps^\pm$ are smooth, bounded and with bounded derivatives, uniformly in $\eps$,  as soon as the variables $\theta$ and $\eps v$ are in a compact.
We obtain 
$$\displaylines{
I_\eps =
  (2\pi)^{-d} \int_{\R^{3d}} a_\eps\left(-x,\Phi(\theta)+{\eps^2\over 2} B_\eps^+(\theta,v)[v,v]\right) {\rm e}^{i x\cdot (d\Phi(\theta) v +\eps B_\eps^-(\theta,v)[v,v])}
 \hfill\cr\hfill
\times\,  {\mathcal F}_\eps f\left(\Phi\left(\theta -\eps {v\over 2}\right)\right)
 \,\overline{{\mathcal F}_\eps f}\left(\Phi\left(\theta +\eps {v\over 2}\right)\right) J_\Phi \left(\theta +\eps {v\over 2}\right)J_\Phi\left(\theta -\eps {v\over 2}\right)  d\theta\,d\theta'\,dx,
\cr}$$
where we have omitted the localization functions in $\theta+\eps {v\over 2}$ and $\theta-\eps {v\over 2}$, which makes that the integral is compactly supported in $\theta$ and $\eps v$ Moreover, we have  $\eps |v|\leq \delta$ on the domain of integration. We shall crucially use this information later.

\medskip

 The change of variable $x=^t d\Phi(\theta)^{-1} u$ gives 
$$\displaylines{
I_\eps =
  (2\pi)^{-d} \int_{\R^{3d}} a_\eps\left(-\,^t d\Phi(\theta)^{-1} u,\Phi(\theta)+{\eps^2\over 2} B_\eps^+(\theta,v)[v,v]\right)
  \hfill\cr\hfill
\times    {\rm e}^{i u\cdot v +i\eps u\cdot \,^t d\Phi(\theta)^{-1} B_\eps^-(\theta,v)[v,v]} 
 {\mathcal F}_\eps f\left(\Phi\left(\theta -\eps{v\over 2}\right)\right)
\,\overline{{\mathcal F}_\eps f}\left(\Phi\left(\theta +\eps{v\over 2}\right)\right) \cr\hfill
\times \,J_\Phi \left(\theta +\eps{v\over 2}\right)J_\Phi\left(\theta -\eps{v\over 2}\right)J_\Phi^{-1} \left(\theta \right) d\theta\,d\theta'\,du,
 \cr}$$
 with the same property on the domain of integration ($\theta$ in a compact and $\eps |v|<\delta$). 
 Note that 
$$\displaylines{
 a_\eps\left(-\,^t d\Phi(\theta)^{-1} u,\Phi(\theta)+{\eps^2\over 2} B_\eps^+(\theta,v)[v,v]\right)\hfill\cr
  = a\bigg(-\, ^t d\Phi(\theta)^{-1} u,\Phi(\theta)+{\eps^2\over 2} B_\eps^+(\theta,v)[v,v], \hfill\cr\hfill {1\over\eps} B\left(\Phi(\theta)+{\eps^2\over 2} B_\eps^+(\theta,v)[v,v]\right) \varphi\left(\Phi(\theta)+{\eps^2\over 2} B_\eps^+(\theta,v)[v,v]\right)\biggr) \cr
   = a\left(-\,^t d\Phi(\theta)^{-1} u,\Phi(\theta), B\left(\Phi(\theta)\right) {\theta'\over\eps}\right) +\eps r_\eps{^{(2)}} (\theta,u,v)[v,v].
\hfill\cr}$$
The matrix $r_\eps^{(2)}$ is supported in a compact independent of $\eps$ in the variables $(u,\theta)$. Besides, the matrix~$r_\eps^{(2)}$  is smooth, bounded, and with bounded derivatives, uniformly in $\eps$,  as soon as the variable $\eps v$ is in a compact, which is the case on the domain of integration of the integral $I_\eps$.
Using Taylor formula on the Jacobian terms, we write 
$$\displaylines{
 a_\eps\left(-\,^t d\Phi(\theta)^{-1} u,\Phi(\theta)+{\eps^2\over 2} B_\eps^+(\theta,v)[v,v]\right)J_\Phi \left(\theta +\eps {v\over 2}\right)J_\Phi\left(\theta -\eps {v\over 2}\right) \hfill\cr
= a\left(-\,^t d\Phi(\theta)^{-1} u,\Phi(\theta), B\left(\Phi(\theta)\right) {\theta'\over\eps}\right) J_\Phi(\theta)^2+\eps r_\eps^{(2)} (\theta,u,v)[v,v]+ \eps r_\eps^{(1)} (\theta,u,v)\cdot v,\cr}$$
where the vector $ r_\eps^{(1)}$  is supported in a compact independent of $\eps$ in the variables $(u,\theta)$ and, as $r_\eps{^{(2)}}$,  is smooth, bounded, and with bounded derivatives, uniformly in $\eps$ on the domain of integration of the integral $I_\eps$ (where $\theta$ is in a compact and $\eps |v|\leq \delta$, $\delta$ to be chosen later).

\medskip 

 Denote by ${\mathcal U}_\eps$ the isometry of $L^2(\R^d)$ :
 $$f^\eps \mapsto  J_\Phi(\cdot)^{d\over 2} {\mathcal F}_\eps f\left(\Phi\left(\cdot \right)\right),$$
 then 
 $$\left(\op_\eps(a_\eps)f\;,\;f\right) = \left(\op_\eps \left(\widetilde a_\eps  \right) {\mathcal U}_\eps f\;,\;{\mathcal U}_\eps f\right)+\eps\left(R_\eps {\mathcal U}_\eps f\;,\;{\mathcal U}_\eps f\right),$$
  with 
 $$\widetilde a_\eps (u,\theta)= a\left(-\,^t d\Phi(\theta)^{-1} u,\Phi(\theta),B\left(\Phi(\theta)\right) {\theta'\over\eps}\right), $$
and 
 where $R_\eps$ is the operator of kernel 
 $$(\theta,\theta')\mapsto (2\pi\eps)^{-d} K_\eps \left({\theta+\theta'\over 2},{\theta-\theta'\over\eps}\right),$$
 with $K_\eps=K_\eps^{(1)}+K_\eps^{(2)}$, 
 $$\displaylines{K_\eps^{(1)}(\theta,v)=  \int_{\R^{d}} \left(r_\eps^{(1)}(\theta, u, v)\cdot v+ r_\eps^{(2)}(\theta,u,v) [v,v]\right)
   {\rm e}^{i u\cdot v +i\eps u\cdot   \,^t d\Phi(\theta)^{-1} B_\eps^-(\theta,v)[v,v]} du,\cr
   K_\eps^{(2)}(\theta,v) =  \int_{\R^{d}} a_\eps\left(^t d\Phi(\theta)^{-1} u,\Phi(\theta)+{\eps^2\over 2} B_\eps^+(\theta,v)[v,v]\right)
   {\rm e}^{i u\cdot v} \cr\hfill 
  \qquad  \times \,{1\over \eps} \left[  {\rm e}^{i\eps u\cdot   \,^t d\Phi(\theta)^{-1} B_\eps^-(\theta,v)[v,v]}-1 \right] du.\qquad
   \cr}$$
 The proof concludes by Schur lemma and the next result. 
\begin{lemma}\label{lem:estrest}
Let us fix $\delta$ small enough. Then, for any $j\in\{1,2\}$, there exists a constant $C_j>0$ such that for all $\eps>0$,
$$ \int_{\R^d} \sup_{\theta\in\R^d} |K_\eps^{(j)}(\theta,v) |dv\leq C_j.$$
\end{lemma}
Indeed, by this Lemma, we obtain that for all $\eps>0$
$$\displaylines{ (2\pi\eps)^{-d}\int_{\R^d} \sup_{\theta\in\R^d}  \left|K_\eps \left({\theta+\theta'\over 2},{\theta-\theta'\over\eps}\right)\right|d\theta'
= (2\pi)^{-d}\int_{\R^d} \sup_{\theta\in\R^d}  |K_\eps \left(\theta-\eps v,v\right)|dv\hfill\cr\hfill 
\leq 
(2\pi)^{-d}\int_{\R^d} \sup_{\theta\in\R^d}  |K_\eps \left(\theta,v\right)|dv
 \leq C_1+C_2,\cr}$$
and similarly 
$$\displaylines{ (2\pi\eps)^{-d}\int_{\R^d} \sup_{\theta'\in\R^d}  \left|K_\eps \left({\theta+\theta'\over 2},{\theta-\theta'\over\eps}\right)\right|d\theta
= (2\pi)^{-d}\int_{\R^d} \sup_{\theta'\in\R^d}  |K_\eps \left(\theta'+\eps v,v\right)|dv\hfill\cr\hfill 
\leq 
(2\pi)^{-d}\int_{\R^d} \sup_{\theta'\in\R^d}  |K_\eps \left(\theta',v\right)|dv
 \leq C_1+C_2,\cr}$$
 By Schur Lemma, these two inequalities yield the boundedness of $R_\eps$ uniformly in $\eps$ as an operator on $L^2(\R^d)$.

\medskip 

Let us now prove Lemma~\ref{lem:estrest}.

\begin{proof}
Note first that the functions $K_\eps^{(j)}$ are compactly supported in the variable $\theta$, uniformly in~$\eps$. We are going to prove that for any $N>0$, there exists a constant $C_{N,j}$ such that, for $| v|>1$, 
$$(1+|v|^2)^N \left| K_\eps^{(j)} (\theta,v)\right| \leq C_{N,j}.$$
These inequalities are enough to conclude as in the lemma. For proving these inequalities, we crucially use that the domain of integration in $u$ is compact and we shall gain the decrease in $v$ by using the oscillations inside the integral.

\medskip

Let us first focus on $K_\eps^{(1)}$. Since $\theta$ is in a compact and $B_\eps^-$ is bounded, we  have 
$$\left| v+\eps \,^td\Phi(\theta)^{-1} B_\eps^-(\theta,v) [v,v]\right|\geq  | v| -M\delta | v|$$
for  some constant $M$. 
Therefore, if $\delta M<1/2$, we have 
$$\left| v+\eps \,^td\Phi(\theta)^{-1} B_\eps^-(\theta,v) [v,v]\right|> {1\over 2} | v|,$$
 and, for $|v|>1$, integration by parts give 
 $$ \displaylines{ K_\eps^{(1)}(\theta,v)= 
  \int_{\R^{d}} \left| v+\eps \,^td\Phi(\theta)^{-1} B_\eps^-(\theta,v) [v,v]\right|^{-2N}\left( \Delta^N_u r_\eps^{(1)}(\theta, u, v)\cdot v+ \Delta^N_u r_\eps^{(2)} [v,v]\right)
\hfill\cr\hfill  
\times\,   {\rm e}^{i u\cdot v +i\eps u\cdot   \,^t d\Phi(\theta)^{-1} 
B_\eps^-{(\theta,v)[v,v]} }
du. 
 \cr}$$
Since $r_\eps^{(1)}$ and $r_\eps^{(2)}$ have smooth compactly supported derivatives in $u$, uniformly bounded in $\eps$, we obtain the existence of a constant $C_{N,1}$ such that 
$$|  K_\eps^{(1)}(\theta,v)| \leq |v|^{-2N} C_{N,1}.$$

\medskip 

Let us now study $K_\eps^{(2)}$ that we turn into 
$$\displaylines{
 K_\eps^{(2)}(\theta,v) =  i \int_0^1\int_{\R^{d}} u \, ^td\Phi(\theta)^{-1} B_\eps ^-(\theta,v) a_\eps\left(^t d\Phi(\theta)^{-1} u,\Phi(\theta)+{\eps^2\over 2} B_\eps^+(\theta,v)[v,v]\right)
 \hfill\cr\hfill \times \,   {\rm e}^{i u\cdot v 
+i t \eps u\cdot   \,^t d\Phi(\theta)^{-1} B_\eps^-(\theta,v)[v,v]} dudt.\cr}$$
Once written on this form, one can see that the arguments developed for $K_\eps^{(1)}$ apply again since the function 
$$u\mapsto u \, ^td\Phi(\theta)^{-1} B_\eps ^-(\theta,v) a_\eps\left(^t d\Phi(\theta)^{-1} u,\Phi(\theta)+{\eps^2\over 2} B_\eps^+(\theta,v)[v,v]\right)$$
is compactly supported in the variable~$u$, smooth and bounded with derivatives that are bounded uniformly in $\eps$. 
\end{proof}

\def\cprime{$'$}

\end{document}